\documentclass[hidelinks,onefignum,onetabnum]{siamart220329}

\usepackage{lipsum}
\usepackage{amsfonts}
\usepackage{graphicx}
\usepackage{epstopdf}
\usepackage{algorithmic}
\ifpdf
  \DeclareGraphicsExtensions{.eps,.pdf,.png,.jpg}
\else
  \DeclareGraphicsExtensions{.eps}
\fi

\newsiamremark{remark}{Remark}
\newsiamremark{hypothesis}{Hypothesis}
\crefname{hypothesis}{Hypothesis}{Hypotheses}
\newsiamthm{claim}{Claim}

\headers{Floquet stability of periodic laser pulses}{V. Shinglot and J. Zweck}

\title{Floquet stability of periodically stationary pulses in a  short-pulse fiber laser\thanks{Submitted to the editors DATE.
\funding{This work was funded by the NSF under DMS-2106203}}}

\author{Vrushaly Shinglot\thanks{Department of Mathematical Sciences, The University of Texas at Dallas, Richardson, TX
  (\email{vrushaly92296@gmail.com,zweck@utdallas.edu)}.}
\and John Zweck\footnotemark[2] 
}

\usepackage{amsopn}

\usepackage{amsmath,amssymb}
\usepackage{physics}
\usepackage{verbatim}

\renewcommand\norm[1]{\left\lVert#1\right\rVert}
\usepackage{csquotes}
\newtheorem*{rem}{Remark}

\ifpdf
\hypersetup{
  pdftitle={Floquet stability of periodically stationary
pulses in a  short-pulse fiber laser},
  pdfauthor={V. Shinglot and J. Zweck}
}
\fi

\newcommand\StartRed{% \StartRed opens group
  \begingroup           % start a group
 \color{red}            % comment this line out for standard black
  \let\EndRed\endgroup % within this 'environment' (only),
}

\begin{document}

\maketitle

% REQUIRED
\begin{abstract}
The quantitative modeling and design of modern short-pulse fiber lasers
cannot be performed with averaged models  because of large variations in the pulse parameters within each round trip. Instead, lumped models obtained by concatenating models for the various components of the laser are required. 
Since the optical pulses in lumped models are periodic,
their linear stability is investigated using the monodromy operator, which is the linearization of the roundtrip operator about the pulse.
A gradient-based optimization method is developed to 
discover periodic pulses.  The computation of the gradient of the objective function involves numerical computation of the action of both the round trip operator and the adjoint of the monodromy operator.
A novel Fourier split-step method is introduced to compute solutions
of the linearization of the nonlinear, nonlocal, stiff equation
that models optical propagation in the fiber amplifier.
This method is derived by  linearizing the two solution operators
in a split-step method for the nonlinear equation.
The spectrum of the monodromy operator consists of the essential
spectrum, for which there is an analytical formula, and the eigenvalues. 
There is a multiplicity two eigenvalue at $\lambda=1$, which is
 due to phase and translation invariance. 
The remaining eigenvalues are determined from a matrix discretization of the monodromy operator. 
Simulation results verify the accuracy of the numerical methods, show examples of periodically stationary pulses, their spectra and eigenfunctions, and discuss their stability.
\end{abstract}

% REQUIRED

% REQUIRED

\section{Introduction}

Since the advent of the soliton laser~\cite{mollenauer1984soliton}, researchers have invented 
several generations of short pulse fiber lasers including dispersion-managed lasers~\cite{kim2014sub,tamura199377},
similariton lasers \cite{fermann2000self, hartl2007cavity},
and the  Mamyshev oscillator \cite{rochette2008multiwavelength,sidorenko2018self,tarasov2016mode}. 
The pulses in these lasers typically have durations on the order of 100~fs,
peak powers on the order of 1~MW, and energy in the 1-50~nJ range.
Applications of femtosecond laser technology include frequency comb generation,
highly accurate measurement of time, frequency, and distance, optical waveform generation, and laser surgery~\cite{diddams2010evolving,fu2018several}.

Traditionally, the modeling  of short pulse lasers has been
based on averaged models,
 in which each of the physical effects that act on the light pulse
is averaged over one round trip of the laser loop to obtain a constant coefficient
partial differential equation such as the cubic-quintic complex Ginzburg-Landau
equation (CQ-CGLE) or the Haus master equation (HME) (see \cite{kutz2006mode} for a review).
This approach has been successfully applied to
soliton lasers for which the pulse  maintains its shape as it propagates.
However, as is highlighted by Turitsyn et al.~\cite{turitsyn2012dispersion}, averaged models cannot be used for
the quantitative modeling and design of  recent generations of 
short pulse lasers due to large variations in the pulse within each round trip.

Instead, the computational modeling of modern short pulse lasers should 
be based on lumped models obtained by concatenating models for the various
components of the laser.  
Typically, short pulse lasers   include an optical fiber amplifier, segments
of  single-mode fiber, a  saturable absorber, a dispersion compensating 
element, a spectral filter, and an output coupler.  
With a lumped model, the pulse  changes shape as it propagates through the various components of the laser system, returning to the same shape
once per round trip. We call such pulses  periodically stationary to distinguish them from the stationary pulses in a soliton laser.

Building on  work of Kaup~\cite{kaup1990perturbation} and Haus~\cite{haus1975theory,haus2000mode},
Menyuk~\cite{menyuk2016spectral}  developed a computational approach to the modeling of stationary pulse solutions of averaged models. 
With this method, stationary pulses are found using a root finding method and their linear stability is determined by computing the spectrum of the
linearization of the governing equation about the pulse. 
While there is an analytical formula for the essential spectrum,
the eigenvalues must be numerically computed, either by solving a 
(possibly nonlinear) eigenproblem
involving a matrix discretization of the differential operator~\cite{shen2016spectra,wang2014boundary}, or by using 
Evans function methods~\cite{bridges2002stability,SINUM53p2329,kapitula2004evans,kapitula1998instability,kapitula1998stability}.  

In this paper, we extend this approach to  periodically stationary pulses in lumped laser models. To keep the presentation concrete, we focus on a
dispersion-managed laser of Kim et al~\cite{kim2014sub}. 
However, the methodology can readily be adapted to other lumped laser models.
First, in \cref{sec:Model}, we describe our lumped model of the Kim laser. The
single-mode fiber segments are modeled by the nonlinear Schr\"odinger equation
(NLSE) and the fiber amplifier is modeled  by the HME, which
is a generalization of the NLSE that includes a nonlocal saturable gain term.
We also introduce the round trip operator, $\mathcal R$, which models 
propagation once around the laser loop, and define a pulse, $\psi$, to be periodically stationary if $\mathcal R \psi= e^{i\theta}\psi$, for some constant phase $\theta$. 
In \cref{sec:Linearization of the Round Trip Operator}, we introduce the monodromy operator, $\mathcal M$, which is the
linearization of $\mathcal R$ about a pulse, $\psi$. 
We formulate the equations for the linearization of each component
of the model, focusing special attention on the linearization of the  HME. Because the nonlinear partial differential equations in the model involve 
the complex conjugate of $\psi$, we choose to define 
$\mathcal M$ to act on $\mathbb R^2$-valued  functions, which should be thought
of as the real and imaginary parts of  $\mathbb C$-valued functions. 

In \cref{sec:PeriodicPulse}, we develop a computational 
method for discovering periodically stationary pulses.
This method, which involves using gradient-based optimization 
to minimize the $L^2$-error between $\mathcal R \psi$ and $e^{i\theta}\psi$,
is an adaptation of a method of Ambrose and Wilkening for computing periodic solutions of partial differential equations~\cite{ambrose2012computing}.
In particular, we provide an analytical formula for the optimal
phase, $\theta$, in terms of the optimal pulse, $\psi$. 
The computation of the gradient of the objective function involves
numerical computation of the action  of both 
the round trip operator, $\mathcal R$, and 
the adjoint, $\mathcal M^*$, of the monodromy operator.

In \cref{sec:SplitStep}, 
we describe the Fourier split-step methods we use to solve the
HME and its linearization. 
For the HME,  we use a method
of Wang et al~\cite{wang2013comparison} designed to handle the frequency filtering term  in the equation,  which is nonlinear, nonlocal, and stiff.
In particular, we provide formulae for
 locally third-order accurate solution operators
for the two steps in the method.
Then, we derive the split-step method for the linearized equation by linearizing these two solution operators. To the best of our knowledge, this approach is novel even in the special case of the linearized NLSE. Finally, the solver for the linearization
is then used to obtain one for its adjoint. The derivation of these methods is not completely straightforward  due to  the nonlocal nature of the saturable gain term in the HME.

In analogy with the Floquet theory of periodic
solutions of ordinary differential equations~\cite{teschl2012ordinary}, 
we expect that the linear stability of a periodically stationary pulse will be determined by the spectrum of the monodromy operator.
The spectrum of $\mathcal M$ is the union of the essential spectrum and the
eigenvalues. In~\cite{shinglot2023essential}, we derived a formula for the essential spectrum.
In \cref{sec:Spectrum}, we show that the monodromy operator has a multiplicity
two eigenvalue at $\lambda=1$, which is due to the phase and time-translation
invariances of $\mathcal R$. These eigenvalues are analogous to the well-known  
eigenvalues of stationary pulses at $\lambda=0$~\cite{kaup1990perturbation}.
As in~\cite{cuevas2017floquet,shinglot2022continuous}, we determine the remaining  eigenvalues from a matrix discretization of $\mathcal M$. 
Finally, in \cref{sec:SimResults} we present simulation results that verify the accuracy of the numerical methods, show examples of periodically stationary pulses, their spectra and eigenfunctions, and discuss their stability.

Some of the results in this paper were announced in~\cite{shinglot2023essential,shinglot2022continuous}. However, the spectra shown here are shown for a new modified version of the monodromy operator introduced in \cref{sec:Spectrum}. 

\section{Mathematical Model}
\label{sec:Model}

In the left panel of \cref{fig:Short pulse laser model}, we show a system diagram 
for the lumped model of the stretched pulse laser of Kim et al.~\cite{kim2014sub}. A light pulse circulates around the loop, passing through a saturable absorber (SA), a segment of single mode fiber (SMF1), a fiber amplifier (FA), a second segment of single mode fiber (SMF2), a dispersion compensation element (DCF), and an output coupler (OC).
After several round trips, the light circulating in the loop forms into a pulse 
that changes shape as it propagates through the different components,
returning to the same shape each time it returns to the same position in the loop.
In the right panel of Figure~\ref{fig:Short pulse laser model} we show the profile
of such a periodically stationary pulse at the output of each component.
The goal of this paper is to study the spectral stability of such 
periodically stationary pulses in lumped  models of fiber lasers.

\begin{figure}[H]
    \centering
    \includegraphics[width=0.3\linewidth]{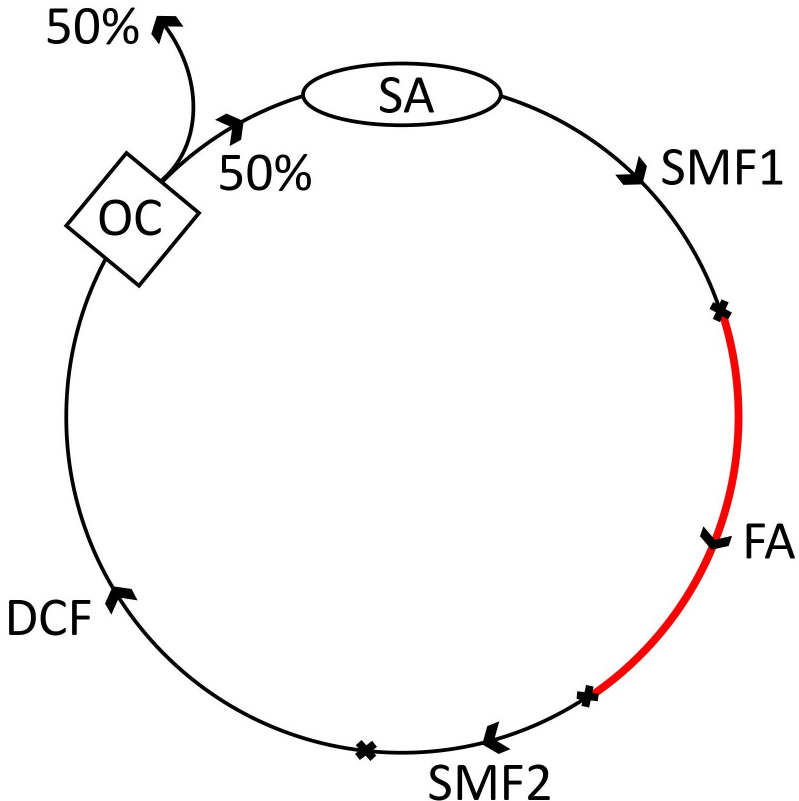}
    \hskip40pt
    \includegraphics[width=0.35\linewidth]{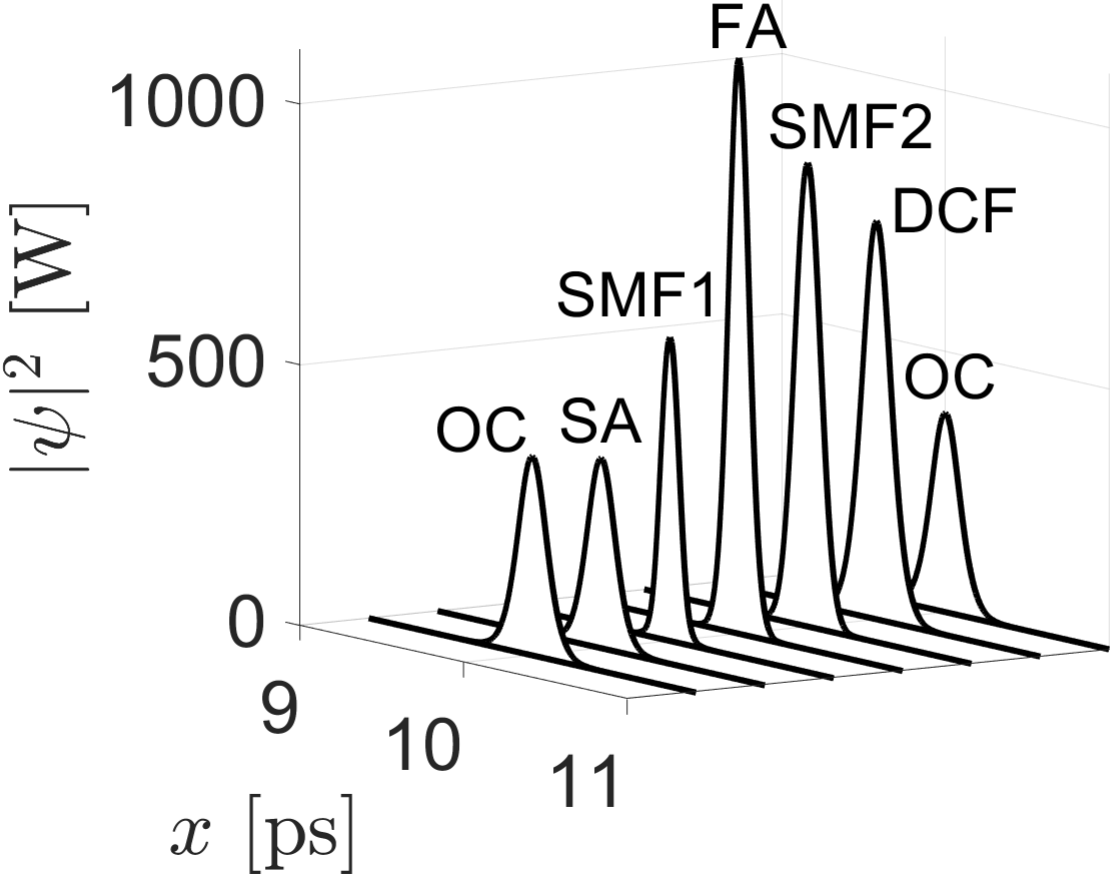}
    \caption{Left: System diagram of the stretched pulse laser of Kim et al.~\cite{kim2014sub}.
    Right: Instantaneous power of the periodically stationary pulse exiting each component of the laser.
    }
    \label{fig:Short pulse laser model}
\end{figure}

At each position in the loop, we model the complex electric field envelope of the light as a function, $\psi = \psi(x)$.
 Physically speaking,  $x$ is a fast-time variable
defined relative to a frame moving at the group
velocity~\cite{menyuk2016spectral,yang2010nonlinear}.
Since the length of the optical fiber in the loop is  on the order of 1~m
and the loop contains a single pulse with  duration  on the order
of 100~fs, the pulse duration is about one ten-thousandth of the round trip
time. Consequently,  it is reasonable to assume that the fast-time variable,
$x$, varies over the entire real line, $\mathbb R$, rather than being periodic. Of course, in numerical computations, we truncate $\mathbb R$ 
to a finite interval.
The pulse is normalized so that $|\psi(x)|^2$ is the instantaneous power. 
 We assume that the function, $\psi$, is an element of the Lebesgue space, $L^2(\mathbb{R}, \mathbb{C})$,
of square integrable, complex-valued functions on $\mathbb R$.
We model each component of the laser as a transfer function, 
$\mathcal{P}:  L^2(\mathbb{R}, \mathbb{C})\rightarrow L^2(\mathbb{R}, \mathbb{C})$, so that
\begin{equation}
    \psi_{\text{out}} = \mathcal{P} \psi_{\text{in}},
    \label{eq:Transfer function}
\end{equation}
where $\psi_{\text{in}}$ and $\psi_{\text{out}}$ are the pulses entering and exiting the component. 
The components in the model come in two flavors: discrete and continuous.  
By a discrete component we mean one in which the action of the operator, 
$\mathcal P$, on the input pulse, $\psi_{\text{in}}$,  is  essentially
obtained in one step, for example by the application of an 
explicit formula. In our model of the Kim laser,  the discrete components are the
saturable absorber, dispersion compensation  element,  and  output coupler.
Short-pulse fiber lasers sometimes also  include a spectral filter
that is modeled as a discrete component. 
By a  continuous component, we mean one in which the action 
of the operator, $\mathcal P$, on the input pulse, $\psi_{\text{in}}$,
is modeled by solving a nonlinear wave equation with initial condition,
$\psi_{\text{in}}$, from the input to the output of the component.
In fiber lasers the continuous components are those 
that involve the propagation of a light pulse through 
a segment of nonlinear optical fiber. For our model of the Kim laser these are
the fiber amplifier and the two segments of single mode fiber.
Note that we have chosen to model the dispersion compensation  element as 
a discrete component, since it is modeled by a constant-coefficient 
linear partial differential equation  which has an analytical solution in the Fourier domain.

With a lumped model, the propagation of a light pulse once around the laser loop
is modeled by the round trip operator, $\mathcal{R}: L^2(\mathbb{R}, \mathbb{C}) \rightarrow L^2(\mathbb{R}, \mathbb{C})$, which is given by the 
composition of the transfer functions of all the components. 
For our model of the Kim laser,  the round trip operator is given by 
\begin{equation}
    \mathcal{R} = \mathcal{P}^{\text{OC}} \circ \mathcal{P}^{\text{DCF}}\circ\mathcal{P}^{\text{SMF2}} \circ\mathcal{P}^{\text{FA}}\circ\mathcal{P}^{\text{SMF1}} \circ \mathcal{P}^{\text{SA}}.
    \label{eq:Round trip operator}
\end{equation}
We say that $\psi_0 \in L^2(\mathbb{R}, \mathbb{C})$ 
is a \emph{periodically stationary pulse} if 
\begin{equation}
 \mathcal{R}(\psi_0) = e^{i\theta} \psi_0,
 \label{eq:PeriodicPulse}
\end{equation}
for some constant phase, $\theta \in [0,2\pi )$. 
For the Kim laser, $\psi_0$ is the pulse at the input to the saturable absorber. For each component, we let $\psi_{\text{in}}$ denote the pulse obtained by propagating the periodically stationary pulse, $\psi_0$, from 
the input of the SA  to the input of that component.
For the continuous fiber 
components we let $\psi$ denote the pulse propagating through that fiber. 

We now describe the model for the propagation of  a light pulse, $\psi= \psi(t,x)$, through the fiber amplifier. 
Here $t$ denotes position along the fiber, with $0\leq t \leq L_{\text{FA}}$, where $L_{\text{FA}}$ is the length of the fiber
amplifier.  We note that  $t$ is a local evolution variable that is 
only defined within the fiber amplifier.  
Our model for propagation in the fiber amplifier is based on the Haus master equation~\cite{haus1975theory}, which is a generalization of the NLSE that includes gain that saturates at high energy and is of finite bandwidth. Specifically, we model the transfer function, $\mathcal{P}^{\text{FA}}$, of a fiber amplifier of length, $L_{\text{FA}}$, as
 $   \psi_{\text{out}} = \mathcal{P}^{\text{FA}}\psi_{\text{in}}$,
 where $\psi_{\text{out}} = \psi({L}_{\text{FA}},\cdot)$ is  obtained by solving the initial value problem
\begin{equation}
\begin{aligned}
    \partial_t \psi 
    &= \left[ \frac{ g(\psi)}{2} \left(1+ \frac{1}{\Omega_g^2}\partial_x^2
    \right) 
    - \tfrac{i}{2}\beta_{\text{FA}} \partial_x^2 
    + i\gamma|\psi|^2\right]\,\psi,
    \qquad \text{for } 0\leq t \leq L_{\text{FA}},
    \\
    \psi(0,\cdot) &= \psi_{\text{in}}.
    \label{eq:fiber amplifier}
\end{aligned}
\end{equation}
Here, $g(\psi)$ is the saturable gain given by
\begin{equation}
    g(\psi) {}={} \frac{g_0}{1+E(\psi)/E_{\text{sat}}},
    \label{eq:Saturable gain}
\end{equation}
where $g_{0}$ is the unsaturated gain, $E_{\text{sat}}$ is the saturation energy, and $E(\psi)$ is the pulse energy, which is given by 
\begin{equation}
    E(\psi) = \int_{\mathbb{R}} |\psi(\cdot,x)|^2 dx.
    \label{eq:PulseEnergy}
\end{equation}
The saturable gain is modeled as a nonlocal function of the  
fast-time variable, $x$, since the response time of the fiber amplifier is 
on the order of 1~ms~\cite{mu2000comparison},
which is much longer than the  pulse duration, which  is on the order of 100~fs.
We also observe that the energy and saturable gain depend on the evolution variable, $t$, since $\psi$ does.
The finite bandwidth of the amplifier is modeled using a Gaussian filter with bandwidth, $\Omega_{g}$. In \eqref{eq:fiber amplifier}, $\beta_{\text{FA}}$ is the chromatic dispersion coefficient and $\gamma$ is the nonlinear Kerr coefficient.

Similarly, we model the transfer function, $\mathcal{P}^{\text{SMF}}$, of a segment of single mode fiber of length, $L_{\text{SMF}}$, as
 $   \psi_{\text{out}} = \mathcal{P}^{\text{SMF}}\psi_{\text{in}}$,
 where $\psi_{\text{out}} = \psi({L}_{\text{SMF}},\cdot)$ is  obtained by solving the initial value problem for the NLSE given by
\begin{equation}
\begin{aligned}
    \partial_t \psi 
    &= 
    - \tfrac{i}{2}\beta_{\text{SMF}}  \partial_x^2 \psi
    + i\gamma|\psi|^2 \psi,
    \qquad \text{for } 0\leq t \leq L_{\text{SMF}},
    \\
    \psi(0,\cdot) &= \psi_{\text{in}}.
    \label{eq:SMF_JZ}
\end{aligned}
\end{equation}
We model the dispersion compensation element as $\mathcal P_{\text{DCF}}
= \mathcal F^{-1} \circ \widehat{\mathcal P}^{\text{DCF}} \circ \mathcal F$, where 
$\mathcal F$ is the Fourier transform and 
 \begin{equation}
 \widehat{\psi}_{\text{out}} (\omega) =    (\widehat{\mathcal{P}}^{\text{DCF}} \widehat{\psi}_{\text{in}})(\omega) = 
    \exp \left( i {\omega^2} \beta_{\text{DCF}}/2  \right)
    \widehat{\psi}_{\text{in}}(\omega),
    \label{eq:DCF_JZ}
\end{equation}
with $\widehat{\psi} = \mathcal F (\psi)$.
We observe that \eqref{eq:DCF_JZ} is the solution to the  initial value problem for the
linear equation obtained by setting $\gamma = 0$, $\beta_{\text{SMF}}=
\beta_{\text{DCF}}$ and $L_{\text{SMF}}=1$ in \eqref{eq:SMF_JZ}.

We model the saturable absorber using the fast saturable loss transfer function~\cite{wang2016comparison}, $\mathcal{P}^{\text{SA}}$, given by
\begin{equation}
    \psi_\text{out} = \mathcal{P}^{\text{SA}}(\psi_{\text{in}}) = \left( 1 - \frac{\ell_0}{1+|\psi_\text{in}|^2/ P_{\text{sat}}}\right)\psi_\text{in},\\
    \label{eq:Saturable absorber}
\end{equation}
where $\ell_0$ is the unsaturated loss and $P_{\text{sat}}$ is the saturation power. With this model,  $\psi_\text{out}$ at  $x$  only depends on $\psi_\text{in}$ at the same value of $x$. Finally, we model the transfer function, $\mathcal{P}^{\text{OC}}$, of the output coupler as 
\begin{equation}
    \psi_{\text{out}} = \mathcal{P}^{\text{OC}}\psi_{\text{in}}
    = \ell_{\operatorname{OC}}\,\psi_{\text{in}},
    \label{eq:Output coupler}
\end{equation}
where  $ (\ell_{\operatorname{OC}})^2$ is the power loss at the output coupler.

%---------------------------------------------------------------------------

%---------------------------------------------------------------------------
\section{Linearization of the Round Trip Operator}
\label{sec:Linearization of the Round Trip Operator}

In this section, we derive formulae for  the 
linearizations, $\mathcal{U}$, about a  pulse 
of each of the  operators, 
$\mathcal{P}$, defined in Section~\ref{sec:Model}. 
By the chain rule, 
the linearization, $\mathcal{M}$, of the round trip operator, $\mathcal{R}$, 
about a periodically stationary pulse, $\psi_0$,
is equal to the composition of the linearized transfer functions, $\mathcal{U}$, 
of each component of the system, i.e., 
\begin{equation}
    \mathcal{M} = \mathcal{U}^{\text{OC}} \circ \mathcal{U}^{\text{DCF}} \circ \mathcal{U}^{\text{SMF2}} \circ \mathcal{U}^{\text{FA}} \circ \mathcal{U}^{\text{SMF1}} \circ \mathcal{U}^{\text{SA}}.
    \label{eq:Linearization of the round trip operator}
\end{equation}
In analogy with the monodromy matrix  in the Floquet theory of periodic solutions of ODE's~\cite{teschl2012ordinary}, we call  $\mathcal M$ the \emph{monodromy operator} of the linearization
of the round trip operator, $\mathcal R$,
 about the periodically stationary pulse, $\psi_0$. 
 In~\cite{shinglot2023essential}, we provide conditions on the smoothness
and decay of the pulse which ensure that the monodromy operator exists on an appropriate Lebesgue function space.

Because the linearization of the partial differential equations in the model involves the complex conjugate of $\psi$, we reformulate the model as a system of equations for the column vector $\boldsymbol{\psi} = [\Re(\psi), \Im(\psi)]^T \in \mathbb R^2$. 
For example, the transfer function of the fiber amplifier
is reformulated as the operator,
 $\mathcal{P}^{\text{FA}} :  L^2(\mathbb{R}, \mathbb{R}^2) \to  L^2(\mathbb{R}, \mathbb{R}^2)$, given by
$ \boldsymbol{\psi}_{\text{out}} = \mathcal{P}^{\text{FA}}\boldsymbol{\psi}_{\text{in}}$, 
where $\boldsymbol{\psi}_{\text{out}} = \boldsymbol{\psi}(\text{L}_{\text{FA}},\cdot)$ is  obtained by solving the initial value problem
\begin{equation}
\begin{aligned}
    \partial_t \boldsymbol{\psi} 
    &= \left[ \tfrac{ g(\boldsymbol{\psi})}{2} \left(1+ \tfrac{1}{\Omega_g^2}\partial_x^2
    \right) 
    - \tfrac{\beta}{2} \textbf{J} \partial_x^2 
    + \gamma \norm{\boldsymbol{\psi}}^2 \textbf{J} \right]\,\boldsymbol{\psi},
    \\
    \boldsymbol{\psi}(0,\cdot) &= \boldsymbol{\psi}_{\text{in}},
    \label{eq:fiber amplifier new}
\end{aligned}
\end{equation}
where $\textbf{J} = \begin{bmatrix} 0&-1\\1&0 \end{bmatrix}$, and $\norm{\cdot}$ is the standard Euclidean norm on $\mathbb{R}^2$.

The linearized transfer function, $\mathcal{U}^{\text{FA}}:  L^2(\mathbb{R}, \mathbb{R}^2) \to  L^2(\mathbb{R}, \mathbb{R}^2)$, in the fiber amplifier 
is given by $\boldsymbol{u}_{\text{out}} = \mathcal{U}^{\text{FA}}\boldsymbol{u}_{\text{in}}$, 
where $\boldsymbol{u}_{\text{out}} = \boldsymbol{u}({L}_{\text{FA}},\cdot)$ is obtained by solving the linearized initial value problem
\begin{equation}
    \begin{aligned}
    \partial_t \boldsymbol{u} &= 
    \mathcal L^{\text{FA}}(\boldsymbol\psi)(\mathbf u) = 
    \left[ g(\boldsymbol{\psi})\mathcal{K} + \mathcal{L} + \mathcal{M}_1(\boldsymbol{\psi}) + \mathcal{M}_2(\boldsymbol{\psi})
    + \mathcal{P}(\boldsymbol{\psi})
    \right] \boldsymbol{u},   \text{ for } 0\leq t \leq L_{\text{FA}},
    \\
    \boldsymbol{u}(0,\cdot) &= \boldsymbol{u}_{\text{in}},
    \end{aligned}
    \label{eq:FA linearized equation}
\end{equation}
where
\begin{equation}
    \begin{aligned}
    \mathcal{K} &= \tfrac{1}{2}\left(1 + \tfrac{1}{\Omega_{g}^{2}}\partial_x^2\right), 
    \qquad 
    &\mathcal{L} &= -\tfrac{\beta}{2}\textbf{J}\partial_{x}^2,\\
    \mathcal{M}_1(\boldsymbol{\psi}) &= \gamma\norm{\boldsymbol{\psi}}^2\textbf{J},
    \qquad 
    &\mathcal{M}_2(\boldsymbol{\psi}) &= 2\gamma\textbf{J}\boldsymbol{\psi}\boldsymbol{\psi}^T,
    \end{aligned}
    \label{eq:Operators in linearized equation}
\end{equation}
and 
\begin{equation}
\mathcal{P}(\boldsymbol{\psi}) \boldsymbol{u} = -\tfrac{g^{2}(\boldsymbol{\psi})}{g_{0}E_{\text{sat}}}\left[\left(1 + \tfrac{1}{\Omega_{g}^{2}}\partial_x^2\right)\boldsymbol{\psi}\right]\int_{-\infty}^{\infty} \boldsymbol{\psi}^T(x) \boldsymbol{u}(x) dx
\label{eq:OperatorP}
\end{equation}
is a nonlocal operator. 
The non-locality of $\mathbf P$, which arises because the gain saturation depends on the total energy of the pulse, makes the analysis more challenging for the fiber
amplifier than for a segment of single mode fiber.
The linearized transfer function, $\mathcal{U}^{\text{SMF}}$, of a segment of single mode fiber  is obtained by setting $g(\boldsymbol\psi) = 0$  in
\eqref{eq:FA linearized equation} and \eqref{eq:OperatorP}.

The linearized transfer function, $\mathcal{U}^{\text{SA}}$,  for the saturable absorber is given by
\begin{equation}
    \boldsymbol{u}_{\text{out}} = \mathcal{U}^{\text{SA}}(\boldsymbol{\psi}_{\text{in}}) \boldsymbol{u}_{\text{in}} = \left( 1 - \ell(\boldsymbol{\psi}_{\text{in}}) - \frac{2 \ell^2(\boldsymbol{\psi}_{\text{in}})}{\ell_0 P_{\text{sat}}} \boldsymbol{\psi}_{\text{in}} \boldsymbol{\psi}_{\text{in}}^T \right) \boldsymbol{u}_{\text{in}},
    \label{eq:SA linearized transfer function}
\end{equation}
where
\begin{equation}
    \ell(\boldsymbol{\psi}_\text{in}) = \frac{\ell_0}{1+\norm{\boldsymbol{\psi}_{\text{in}}}^2/ P_{\text{sat}}}.
    \label{eq:l(psi) in SA transfer function}
\end{equation}
The remaining components, i.e. dispersion compensation fiber and output coupler, already have linear transfer functions, and so $\mathcal U^{\text{DCF}} = \mathcal P^{\text{DCF}}$ and $\mathcal U^{\text{OC}} = \mathcal P^{\text{OC}}$.

Because eigenvalues and eigenfunctions can be complex valued, we extend the linearized system to act on complex-valued functions, $\boldsymbol{u} \in L^2(\mathbb{R}, \mathbb{C}^2)$, where 
\begin{equation}
    L^2(\mathbb{R}, \mathbb{C}^2) = \{ \boldsymbol{u} = \boldsymbol{v} + i \boldsymbol{w} \,\,: \,\, \boldsymbol{v}, \boldsymbol{w} \in L^2(\mathbb{R}, \mathbb{R}^2) \},
\end{equation}
is the space of $\mathbb{C}^2$-valued functions on $\mathbb{R}$ with the standard Hermitian inner product.
Let $\mathcal T$ be an operator that acts on  $\mathbb{R}^2$-valued functions.
We extend $\mathcal T$ to act on $\mathbb{C}^2$-valued functions by defining
$  \mathcal{T} \boldsymbol{u} = \mathcal{T} \boldsymbol{u}_1 + i \mathcal{T} \boldsymbol{u}_2$,
where $\boldsymbol{u} = \boldsymbol{u}_1 + i \boldsymbol{u}_2$ with $\boldsymbol{u}_1, \boldsymbol{u}_2 \in L^2(\mathbb{R}, \mathbb{R}^2)$. 
Note that the formulae above 
for the action of the differential operators and transfer functions 
on $\mathbb{C}^2$-valued functions, $\boldsymbol{u}$,
are the same as for their action
on $\mathbb{R}^2$-valued functions, since in both cases we only require
$\boldsymbol\psi$ to be  $\mathbb{R}^2$-valued. The only difference is our
interpretation of the function spaces on which they act.

\section{Computation of Periodically Stationary Pulses}
\label{sec:PeriodicPulse}

We formulate the problem of finding periodically stationary pulses as that of finding a zero of the Poincar\'e map functional, 
  $\mathcal{E}:L^2(\mathbb{R}, \mathbb{R}^2) \times [0,2\pi ) \rightarrow\mathbb{R}$, 
 given by
\begin{equation}
    \mathcal{E}(\boldsymbol\psi_0, \theta) = \frac{1}{2} \norm{\mathcal{R}(\boldsymbol\psi_0) - \mathbf R(\theta) \boldsymbol\psi_0}^2_{L^2(\mathbb{R}, \mathbb{R}^2)},
    \label{eq:Poincare map functional}
\end{equation}
where $\mathbf R(\theta)$ is the  rotation matrix on $\mathbb R^2$ that corresponds to the operator of
multiplication by $e^{i\theta}$ on $\mathbb C$.\footnote{We note that, for a given set of system parameters, there is no guarantee that 
a periodically stationary pulse exists.}

Next,  we describe the two-stage method 
we use to compute periodically stationary pulse solutions, $\boldsymbol\psi_0$, of the laser system model in \cref{sec:Model}. In the first (evolutionary) stage we propagate 
a Gaussian pulse over sufficiently many round trips of the laser to obtain a good
initial guess for the second (optimization) stage. In the optimization stage, we 
use a gradient-based  method to minimize the objective function
given by the ratio of the Poincar\'e map functional~\eqref{eq:Poincare map functional} and the pulse energy, \eqref{eq:PulseEnergy},
\begin{equation}
\widetilde{\mathcal E}(\psi_0,\theta)\,\,= \,\, \frac{\mathcal E(\psi_0,\theta)}{E(\psi_0)}.
\label{eq:PoinObjFun}
\end{equation}
We note that if $\psi_0$ is a nonzero periodically stationary pulse, then
there is a $\theta$ so that
$\widetilde{\mathcal E}$ has a global minimum value of zero 
at $(\psi_0,\theta)$.
Therefore, to find nontrivial periodically stationary pulses it makes sense to use an optimization algorithm to drive $\widetilde{\mathcal E}$ to zero.
In parameter continuation studies, the first stage
can be omitted if the optimal pulse computed with the previous set of system parameters is a good enough initial guess for optimization with the current
set of parameters. 

In the following theorem,  we adapt a method of Ambrose and Wilkening~\cite{ambrose2012computing} 
for computing the gradient of $\mathcal E$ with respect to the pulse.
With this method, the cost of
computing a directional derivative of $\mathcal E$ is comparable to that
of propagating a pulse and its linearization for one round trip of the laser. 

\begin{theorem}
\label{thm:GradientPoincareMap}
%\label{thm:Gradient of Poincare map wrt the pulse}
The variational derivative of $\mathcal{E}$ with respect to $\boldsymbol\psi_0$ is given by
\begin{equation}
    D_{\boldsymbol\psi_0} \mathcal{E}(\mathbf u_0) =  \left \langle \frac{\delta \mathcal{E}}{\delta \boldsymbol\psi_0}, \,\, \mathbf  u_0 \right \rangle_{L^2(\mathbb{R}, \mathbb{R}^2)},
    \label{eq:Variational derivative of Poincare map}
\end{equation}
where
\begin{equation}
    \frac{\delta \mathcal{E}}{\delta \boldsymbol\psi}(\boldsymbol\psi_0) = \mathcal{M}^*(\mathbf v_0) - \mathbf R(-\theta) \mathbf v_0,
    \label{eq:Gradient of Poincare map wrt pulse}    
\end{equation}
where $\mathbf  v_0 := \mathcal{R}(\boldsymbol\psi_0) - \mathbf R(\theta) \boldsymbol\psi_0$ 
 is a measure of how far $\boldsymbol\psi_0$ is from being periodically stationary,
and  the adjoint of $\mathcal M$ is given by
\begin{equation}
    \mathcal{M}^* = \left(\mathcal{U}^{\emph{SA}}\right)^* \circ \left(\mathcal{U}^{\emph{SMF1}}\right)^* \circ \left(\mathcal{U}^{\emph{FA}}\right)^* \circ \left(\mathcal{U}^{\emph{SMF2}}\right)^* \circ \left(\mathcal{U}^{\emph{DCF}}\right)^* \circ \left(\mathcal{U}^{\emph{OC}}\right)^*,
\end{equation}
where for each component $\mathcal{U}^*$ is the adjoint of the corresponding operator $\mathcal{U}$.

In a fiber segment of length, $L$,  the adjoint of the linearized solution operator, $\mathcal U$, for the fiber is given by
\begin{equation}
    \mathbf  v_L = \mathcal{U}^* \mathbf v_0,
\end{equation}
with $\mathbf v_L = \mathbf v(L,\cdot)$. Here, $\mathbf v = \mathbf v(s,\cdot)$ is obtained by solving the adjoint linearized initial value problem given by
\begin{equation}
\begin{aligned}
     \partial_s \mathbf v(s,\cdot) &= \mathcal{L}^*(\boldsymbol\psi(L-s,\cdot)) 
    \mathbf v(s,\cdot),\\
     \mathbf v(0,\cdot) &= \mathbf v_0,
\end{aligned}
\label{eq:Adjoint linearized IVP}
\end{equation}
where $\mathcal{L}^*(\boldsymbol\psi)$ is the adjoint of the linearized differential operator $\mathcal{L}(\boldsymbol\psi)$,  as in \eqref{eq:FA linearized equation}. 
\end{theorem}

\begin{rem}
Note that here we have $s = L-t$ so that solving the adjoint equation from $s=0$ to $s=L$ propagates the initial pulse $\mathbf v_0$ backwards in $t$ from $t=L$ to $t=0$.
The formula for the operator, $\mathcal{L}^*$, in a fiber segment is obtained
from the formula for $\mathcal{L}$ in \eqref{eq:FA linearized equation} by taking the transposes of all matrices. The operator, $\mathcal{U}^{\emph{SA}}$, is self-adjoint. 
\end{rem}

\begin{proof} The variational derivative of $\mathcal{E}$ with respect to $\boldsymbol\psi_0$ is given by
\begin{align}
    D_{\boldsymbol\psi_0} \mathcal{E}(\mathbf u_0) &= \lim_{\epsilon \to 0} \frac{1}{\epsilon}\left(\mathcal{E}(\boldsymbol\psi_0+\epsilon \mathbf u_0,\theta)-
    \mathcal{E}(\boldsymbol\psi_{0},\theta)\right)
    \\
         &= \langle \mathcal{R}(\boldsymbol\psi_0) - \mathbf R (\theta)\boldsymbol \psi_0, \mathcal{M}(\mathbf u_0) - \mathbf R (\theta) \mathbf  u_0 \rangle_{L^2(\mathbb{R}, \mathbb{R}^2)},
 \label{eq:Derivative of P} 
    \end{align}
where we have used the fact that $\mathcal{R}(\psi_0 + \epsilon u_0) \approx \mathcal{R}(\psi_0) + \epsilon \mathcal{M}(u_0)$.
Setting $\mathbf v_0 := \mathcal{R}(\boldsymbol\psi_0) - \mathbf R(\theta)\boldsymbol\psi_0$ we find that
\begin{equation}
    D_{\boldsymbol\psi_0} \mathcal{E}(\mathbf u_0) =   \langle \mathcal{M}^*(\mathbf v_0) - \mathbf R(-\theta)\mathbf v_0, 
  \mathbf u_0 \rangle_{L^2(\mathbb{R}, \mathbb{R}^2)},
 \label{eq:Derivative of P new}
\end{equation}
which proves~\eqref{eq:Variational derivative of Poincare map}.

To derive~\eqref{eq:Adjoint linearized IVP}, we invoke the defining formula for
$\mathcal U^*$, 
\begin{equation}
    \langle\mathbf  v_0, \mathbf u_L \rangle = 
    \langle \mathbf v_0, \mathcal{U}(\mathbf u_0) \rangle = 
    \langle \mathcal{U}^*(\mathbf v_0), \mathbf u_0 \rangle =
     \langle \mathbf v_L, \mathbf u_0 \rangle.
   \label{eq:Inner product equality}
\end{equation}
Next, we set $s = L-t$ and introduce a function $\mathbf v = \mathbf v(s,x)$ to be chosen so that
\begin{equation}
    h(s) = \langle \mathbf v(s,\cdot),\,\, \mathbf u(L-s, \cdot) \rangle_{L^2(\mathbb{R}, \mathbb{R}^2)}
\end{equation}
is constant.
 Then $\mathbf v_L = \mathbf v(L,\cdot)$  will satisfy \eqref{eq:Inner product equality} as required.
 To derive an equation for $\mathbf v$, we differentiate $h$ to obtain
\begin{align*}
    h^{\prime}(s) &= \langle \partial_s \mathbf v(s,\cdot),\,\, \mathbf u(L-s,\cdot) \rangle_{L^2(\mathbb{R}, \mathbb{R}^2)} - \langle \mathbf v(s,\cdot), 
    \,\, \partial_t u(L-s,\cdot) \rangle_{L^2(\mathbb{R}, \mathbb{R}^2)}\\
    &= \langle \partial_s \mathbf v(s,\cdot),\,\, \mathbf u(L-s,\cdot)\rangle_{L^2(\mathbb{R}, \mathbb{R}^2)} - \langle \mathbf v(s,\cdot), 
    \,\,\mathcal{L}(\boldsymbol\psi(L-s,\cdot)) \mathbf u(L-s,\cdot)\rangle_{L^2(\mathbb{R}, \mathbb{R}^2)}\\
    &= \langle \partial_s \mathbf v(s,\cdot) - \mathcal{L}^*(\boldsymbol\psi(L-s,
    \cdot)) \mathbf v(s,\cdot),\,\, \mathbf u(L-s,\cdot)\rangle_{L^2(\mathbb{R}, \mathbb{R}^2)},
\end{align*}
which is zero provided that $\mathbf v$ satisfies the initial value problem \eqref{eq:Adjoint linearized IVP}.
\end{proof}

Next, we derive an analytical formula for the derivative of $\mathcal E$ with respect to $\theta$. For this result, it is easier to work over $\mathbb C$ than
$\mathbb R^2$.

\begin{proposition}
Suppose that $(\psi_0, \theta)$ is a local minimum of 
\begin{equation}
\mathcal{E}(\psi_0,\theta) =  
\frac 12 \| \mathcal R(\psi_0) - e^{i\theta}\psi_0 \|_{L^2(\mathbb R, \mathbb C)}.
\label{eq:DefEcomplex}
\end{equation}
Then $\theta = \theta(\psi_0)$ is given in terms of $\psi_0$ by
\begin{equation}
    (\cos{\theta}, \sin{\theta}) = \frac{1}{\sqrt{G^2(\psi_0) + H^2(\psi_0)}}(G(\psi_0), H(\psi_0)),
\end{equation}
where 
\begin{equation}
\begin{aligned}
    &F(\psi_0) = \frac{1}{2}\left\{ \norm{\mathcal{R}(\psi_0)}_{L^2(\mathbb{R}, \mathbb{C})}^2 + \norm{\psi_0}_{L^2(\mathbb{R}, \mathbb{C})}^2 \right\},\\
    &G(\psi_0) = \Re \langle \mathcal{R}(\psi_0), \psi_0 \rangle, \quad H(\psi_0) = \Im \langle \mathcal{R}(\psi_0), \psi_0 \rangle.
\end{aligned}
\end{equation}
Let
\begin{equation}
    {\mathcal{F}}(\psi_0) := \mathcal{E}(\psi_0, \theta(\psi_0)) = F(\psi_0) - \sqrt{G^2(\psi_0) + H^2(\psi_0)}.
    \label{eq:Poincare map functional 2}
\end{equation}
Then
\begin{equation}
    \frac{\delta {\mathcal{F}}}{\delta \psi}(\psi_0) = \frac{\delta \mathcal{E}}{\delta \psi}(\psi_0, \theta(\psi_0)),
    \label{eq:EDersEqual}
\end{equation}
where $\frac{\delta \mathcal{E}}{\delta \psi}(\psi_0, \theta(\psi_0))$ is given by \eqref{eq:Gradient of Poincare map wrt pulse}. Furthermore, $(\psi_0, \theta)$ is a local minimum of $\mathcal{E}$ iff $\psi_0$ is a local minimum of ${\mathcal{F}}$.
\end{proposition}

\begin{proof}
By \eqref{eq:DefEcomplex},
\begin{align}
    \mathcal{E}(\psi_0, \theta) 
    &\,\,=\,\, \frac{1}{2}\left\{ \norm{\mathcal{R}(\psi_0)}_{L^2(\mathbb{R}, \mathbb{C})}^2 + \norm{\psi_0}_{L^2(\mathbb{R}, \mathbb{C})}^2 \right\} - \Re \langle e^{-i\theta} \mathcal{R}(\psi_0), \psi_0 \rangle_{L^2(\mathbb{R}, \mathbb{C})}
    \nonumber
    \\
    &\,\,=\,\, F(\psi_0) - \left( G(\psi_0) \cos{\theta} + H(\psi_0) \sin{\theta} \right).
\label{eq:EwrtFGH}
\end{align}    
Therefore, $\frac{\partial \mathcal{E}}{\partial \theta} = 0$ iff
\begin{equation}
    (\cos{\theta}, \sin{\theta}) = \frac{\pm 1}{\sqrt{G^2(\psi_0) + H^2(\psi_0)}}(G(\psi_0), H(\psi_0)).
    \label{eq:Critical point 1}
\end{equation}
To determine which of the signs in \eqref{eq:Critical point 1} 
corresponds to a local minimum of $\mathcal{E}(\theta)$ we observe that 
when $\theta$ satisfies \eqref{eq:Critical point 1}, 
the second derivative of $\mathcal E$ is given by
\begin{equation}
        \frac{\partial^2 \mathcal{E}}{\partial \theta^2} = G(\psi_0) \cos{\theta} + H(\psi_0) \sin{\theta} = \pm \sqrt{G^2(\psi_0) + H^2(\psi_0)}.
\end{equation}
Substituting the value of $\theta$ given by \eqref{eq:Critical point 1}
with the  $+$ sign into \eqref{eq:EwrtFGH}, we obtain
\eqref{eq:Poincare map functional 2}.
Finally, since
\begin{equation}
    \frac{\delta {\mathcal{F}}}{\delta \psi}(\psi_0) = \frac{\delta \mathcal{E}}{\delta \psi}(\psi_0, \theta(\psi_0)) + \frac{\delta \mathcal{E}}{\delta \theta}(\psi_0, \theta(\psi_0)) \frac{\delta \theta}{\delta \psi}(\psi_0)
\end{equation}
and  $\frac{\delta \mathcal{E}}{\delta \theta}(\psi_0, \theta(\psi_0)) = 0$, we obtain
\eqref{eq:EDersEqual}.
\end{proof}

\section{Fourier split-step method}
\label{sec:SplitStep}

In this section, we describe the Fourier split-step schemes we use to solve for the
nonlinear propagation of the pulse, $\boldsymbol\psi$, and its linearization, 
$\mathbf u$, in the fiber segments. These methods are based on the well known symmetric split-step scheme for the NLSE, which is globally second order accurate~\cite{weideman1986split}. 
Wang et al~\cite{wang2013comparison} show that, in addition to being nonlinear,
 the frequency filtering term,  $g(\boldsymbol\psi)\boldsymbol\psi_{xx}$,
 in the fiber amplifier
 equation~\eqref{eq:fiber amplifier}  is stiff. 
 Therefore,  we  make use of a numerical method they designed to handle this stiff term. 
With this method, we propagate the pulse for one time step
with the aid of a  frequency domain solution operator for the
stiff frequency filtering and chromatic dispersion  terms  and of a 
fast-time domain 
solution operator for the Kerr nonlinearity term.
We then derive a split-step method for the linearized equation \eqref{eq:FA linearized equation}
by linearizing these two solution operators. 
This approach yields explicit locally third-order accurate analytical formulae that do not involve the numerical computation of integrals over time. 
To the best of our knowledge, this
approach is novel even in the special case of the linearized NLSE.

\subsection{Operator splitting}

The level of rigor in discussions of  the symmetric split-step Fourier method for
nonlinear wave equations  varies widely~\cite{agrawal2006nonlinear,bao2002time,lu2015strang,muslu2003split,sinkin2003optimization,weideman1986split,yang2010nonlinear}. At one end of the spectrum is 
the rigorous convergence result of Lubich~\cite{lubich2008splitting}. 
At the other end are discussions that do not even explicitly address the sense in which the solution of $\partial_t \mathbf f = \mathcal C(t) \mathbf f$ is $\mathbf f(t) = \exp(\int_0^t \mathcal C(s)\, ds ) \mathbf f(0)$. Here we are thinking of $\mathcal C(t)$
as being the differential operator on the right hand side of the equation.
Indeed, equality  is not guaranteed to hold unless 
$\mathcal C(t_1)\,\mathcal C(t_2) = \mathcal C(t_2)\,\mathcal C(t_1)$ 
for all $t_1, t_2$, which is not even true  in the  case of the NLSE.
To provide an accessible explanation as to why the symmetric split-step Fourier methods for the fiber amplifier equation~\eqref{eq:fiber amplifier}  and its linearization~\eqref{eq:FA linearized equation} are  locally third-order accurate we begin with a discussion of operator splitting in this context.

\begin{proposition}
The solution to $\partial_t \mathbf f = \mathcal C(t) \mathbf f$ is of the form
\begin{equation}
\mathbf f(t+h) = \exp(\int_t^{t+h} \mathcal C(s)\, ds ) \mathbf f(t)\,\,+\,\,
\mathcal O(h^3).
\end{equation}
\end{proposition}

\begin{proof}
For simplicity, we assume $t=0$. Substituting 
\begin{align}
\mathbf f(h) &= \mathbf f_0 + h \mathbf f_1 + h^2 \mathbf f_2 + \mathcal O(h^3),
\\
\mathcal C(h) &= \mathcal C_0 + h \mathcal C_1 + h^2 \mathcal C_2 + \mathcal O(h^3),
\end{align}
into the differential equation and equating coefficients of $h$, we find that 
\begin{align}
\mathbf f(h) &= 
\mathbf f_0 + h \mathcal C_0 \mathbf f_0 + \tfrac 12 h^2 
\left( \mathcal C_0^2 \mathbf f_0  + \mathcal C_1 \mathbf f_0 \right)
+ \mathcal O(h^3)
\\
&= \left[ \exp(\mathcal C_0 h) + \tfrac 12 h^2 \mathcal C_1 \right] \mathbf f_0
+ \mathcal O(h^3)
\\
&= \exp(\int_0^{h} \mathcal C(s)\, ds ) \mathbf f_0 
+ \mathcal O(h^3).
\end{align}
\end{proof}

The nonlinear and linearized equations in the fiber amplifier are both 
of the form
\begin{equation}
\partial_t \mathbf f \,\,=\,\, (\mathcal A(t) + \mathcal B(t)) \mathbf f,
\label{eq:SplitStepDE}
\end{equation}
where for the nonlinear equation (with $\mathbf f = \boldsymbol\psi$),
\begin{equation}
\mathcal A(t) = \mathcal L + g( \boldsymbol\psi(t))\mathcal K
\qquad\text{and}\qquad
\mathcal B(t) = \mathcal M_1(\boldsymbol\psi(t)),
\end{equation}
and for the linear equation (with $\mathbf f = \mathbf u$),
\begin{equation}
\mathcal A(t) = \mathcal L + g( \boldsymbol\psi(t))\mathcal K + 
 \mathcal P(\boldsymbol\psi(t))
\qquad\text{and}\qquad
\mathcal B(t) = \mathcal M_1(\boldsymbol\psi(t)) +
 \mathcal M_2(\boldsymbol\psi(t)).
\end{equation}
Let 
\begin{gather}
\mathcal A_1(h) := \int_t^{t+h/2} \mathcal A(s)\, ds = \widetilde{\mathcal A}_1 h/2
\qquad
\mathcal A_2(h) := \int_{t+h/2}^{t+h} \mathcal A(s)\, ds = \widetilde{\mathcal A}_2 h/2
\label{eq:A1A2}
\\
\mathcal B_0(h) := \int_t^{t+h} \mathcal B(s)\, ds = \widetilde{\mathcal B} h,
\label{eq:B0}
\end{gather}
where the final equalities follow from the mean value theorem for integrals.
In the special case of the NLSE, $\mathcal A(t)=\mathcal A$ 
is $t$-independent, and so the  symmetric split-step scheme
\begin{align}
\boldsymbol\psi(t+h) &= \exp(\mathcal Ah + \mathcal B_0(h))\boldsymbol\psi(t)
+ \mathcal O(h^3) \nonumber
\\
&= \exp(\mathcal Ah/2) \exp(\mathcal B_0(h)) \exp(\mathcal Ah/2)\boldsymbol\psi(t) + \mathcal O(h^3),
\end{align}
holds by two applications of the Baker-Campbell-Haussdorff formula
\begin{equation}
\exp(\mathcal Xh)\exp(\mathcal Yh) 
= 
\exp((\mathcal X + \mathcal Y)h + \frac 12 [\mathcal X, \mathcal Y] h^2) 
+ \mathcal O(h^3), 
\label{eq:BCH}
\end{equation}
where  
$[\mathcal X, \mathcal Y]  = \mathcal X \mathcal Y  - \mathcal Y \mathcal X$.
We note however, that for general operators,  $\widetilde{\mathcal A}_1$ and
 $\widetilde{\mathcal A}_2$, 
 \begin{equation}
 \exp( [\widetilde{\mathcal A}_2 + \widetilde{\mathcal B}  + \widetilde{\mathcal A}_1] h) 
 \neq 
 \exp(\widetilde{\mathcal A}_2 h) 
 \exp(\widetilde{\mathcal B} h) 
 \exp(\widetilde{\mathcal A}_1 h) 
 + \mathcal O(h^3).
 \label{eq:SymmFails}
 \end{equation}
 Nevertheless, we will now show that equality holds 
in \eqref{eq:SymmFails} for the operators in \eqref{eq:A1A2}.
Keeping only terms of order $< h^3$, we find that in general
\begin{equation}
\exp(\mathcal A_2) \exp(\mathcal B_0)\exp(\mathcal A_1)
= \exp(\mathcal A_1+\mathcal A_2+\mathcal B_0 + 
\frac 12 [\mathcal B_0, \mathcal A_1- \mathcal A_2] + 
[\mathcal A_2,\mathcal A_1]).
\end{equation}
Therefore,  it suffices to show that for the operators in \eqref{eq:A1A2},
$\mathcal A_1- \mathcal A_2 = \mathcal O(h^2)$ and
$[\mathcal A_2,\mathcal A_1] = \mathcal O(h^3)$. 
Using Taylor series, these results follow from the formulae
\begin{align}
\mathcal A_1(h) \,\,&=\,\, \tfrac h2 \left[ \mathcal A(t) + \tfrac h4 \mathcal A'(t)\right]
+ \mathcal O(h^3) 
\label{eq:A1approx}
\\
\mathcal A_2(h) \,\,&=\,\, \tfrac h2  \left[ \mathcal A(t+\tfrac h2) + \tfrac h4 \mathcal A'(t
+ \frac h2)\right]
+ \mathcal O(h^3) \\
\mathcal A_2(h) - \mathcal A_1(h) 
\,\,&=\,\, \tfrac h2 \left[ \tfrac h2 \mathcal A'(t) + \tfrac{h^2}{8}\mathcal A''(t)\right]
+ \mathcal O(h^3).
\end{align}
To summarize, the symmetric split-step scheme for \eqref{eq:SplitStepDE}
is given by
\begin{equation}
\mathbf f(t+h) = \exp(\mathcal A_2(h)) \exp(\mathcal B_0(h)) \exp(\mathcal A_1(h))\,\mathbf f(t) + \mathcal O(h^3).
\label{eq:SSS}
\end{equation}
For greater computational efficiency we 
use Richardson extrapolation to combine solutions with step sizes of $h$, $h/2$, 
and $h/4$ to obtain the globally fourth-order accurate
scheme,
\begin{equation}
\mathbf f_{k} 
= \frac{1}{21} \left[
32 \mathbf f_{k}^{h/4} - 12 \mathbf f_{k}^{h/2} + \mathbf f_{k}^{h}\right],
\end{equation}
where $ \mathbf f_{k}^{h}$ is the solution at time step $k$ obtained using \eqref{eq:SSS} with a step size of $h$.

\subsection{Solution operators for the nonlinear equations}

Next, we state two propositions that give analytical formulae for the two solution operators for the nonlinear equation~\eqref{eq:fiber amplifier new} in the fiber amplifier. Setting $g=0$
gives the corresponding results for the single mode fiber segments.

\begin{proposition}\label{prop:KerrTerm}
The solution operator for the Kerr nonlinearity term,
\begin{equation}
\partial_t \boldsymbol\psi \,\,=\,\, \gamma \|  \boldsymbol\psi  \|^2 \mathbf J
\,\boldsymbol\psi,
\label{eq:KerrTerm}
\end{equation}
in the nonlinear equation~\eqref{eq:fiber amplifier new} for the fiber amplifier is given by
\begin{equation}
\boldsymbol\psi (t+h,x) 
\,\,=\,\,
\exp( \gamma   \int_{t}^{t+h}
\| \boldsymbol\psi (s,x) \|^2 \,\mathbf J\, ds) \, \boldsymbol\psi (t,x)
\,\,=\,\,
\mathbf R( \gamma \| \boldsymbol\psi (t,x) \|^2 h)\boldsymbol\psi (t,x)
\label{eq:KerrSolution}
\end{equation}
where 
\begin{equation}
\mathbf R(b) = \begin{bmatrix} \cos b & -\sin b\\ \sin b & \cos b \end{bmatrix}.
\end{equation}
\end{proposition}

\begin{proof}
Applying~\eqref{eq:KerrTerm}, we see that $\| \boldsymbol\psi(s,x) \|^2$ is constant in $s$.
The result now follows from the fact that 
\begin{equation}
\exp(a\mathbf I+ b\mathbf J) \,\,=\,\, e^a\,\mathbf R(b).
\label{eq:expJ}
\end{equation}
\end{proof}

\begin{proposition}
The solution operator for the term,
\begin{equation}
\partial_t \boldsymbol\psi \,\,=\,\, 
\left(\mathcal L + g( \boldsymbol\psi(t,\cdot))\mathcal K
\right)
\,\boldsymbol\psi,
\label{eq:HMELinearTerm}
\end{equation}
in the nonlinear equation~\eqref{eq:fiber amplifier new} for the fiber amplifier is given  by
\begin{align}
\boldsymbol\psi (t+h/2,x) 
\,\,&=\,\,
\exp(  \int_{t}^{t+h/2}
\mathcal L + g( \boldsymbol\psi(s,\cdot))\mathcal K\, ds ) \,
\boldsymbol\psi (t,x)
\nonumber
\\
\,\,&=\,\,
\mathcal F^{-1} \left( e^{G(t,t+h/2)a(\omega)} \,
\mathbf R( b(\omega)h/2) 
\,
\widehat{\boldsymbol\psi} (t,\omega)\right),
\label{eq:LpgKsolution}
\end{align}
where $\mathcal F$ is the Fourier transform,
\begin{equation}
a(\omega) \,\,=\,\, \tfrac 12 \left(1 - \tfrac{\omega^2}{\Omega_g^2}\right)
\qquad\text{and}\qquad
b(\omega) \,\,=\,\, \tfrac 12 \beta\omega^2,
\end{equation}
 and
\begin{equation}
G(t,t+h/2) \,\,=\,\, \int_t^{t+h/2} g(\boldsymbol\psi(s))\, ds.
\label{eq:Gexact}
\end{equation}
Finally, to compute $\boldsymbol\psi(t+h/2,\cdot)$ only
in terms of $\boldsymbol\psi(t,\cdot)$ we employ the
approximation
\begin{equation}
G(t,t+h/2)\,\,=\,\, \tfrac h2 \left( g(t) + \tfrac h4 g_2(t) \right ) + \mathcal O(h^3),
\label{eq:Gapprox}
\end{equation}
where $g(t)=g(\boldsymbol\psi(t,\cdot))$ is given by \eqref{eq:Saturable gain} and 
\begin{equation}
g_2(t) \,\,:=\,\, g'(t)\,\,=\,\,
\frac{-2 g^2(t)}{g_0E_{\emph{sat}}}  \Re
\int_{\\-\infty}^\infty [\widehat{\boldsymbol\psi}(t,\omega)]^*
\left( b(\omega)\mathbf J  + g(t)  a(\omega) \mathbf I \right)
\widehat{\boldsymbol\psi}(t,\omega)\, d\omega.
\label{eq:gprime}
\end{equation}
where $\mathbf v^*$ denotes
the conjugate transpose of a column vector $\mathbf v^*$.
\end{proposition}

\begin{proof}
Equation~\eqref{eq:LpgKsolution} follows from the fact that
\begin{equation}
\mathcal L + g( \boldsymbol\psi(s))\mathcal K 
\,\,=\,\, 
\mathcal F^{-1}\,\circ\, 
\left( b(\omega) h\, \mathbf J \,\,+\,\, G(t,t+h)a(\omega)\,\mathbf I \right)
\, \circ \mathcal F,
\end{equation}
and then applying \eqref{eq:expJ}.
The derivation of \eqref{eq:Gapprox} is the same as that of
\eqref{eq:A1approx}.
Finally, \eqref{eq:gprime} follows from \eqref{eq:Saturable gain}, the formula
\begin{equation}
E'(t) \,\,=\,\, 2\Re \int_{-\infty}^\infty [\widehat{\boldsymbol\psi}(t,\omega)]^* 
\partial_t \widehat{\boldsymbol\psi}(t,\omega)\, d\omega,
\end{equation}
for the derivative of the pulse energy, and \eqref{eq:HMELinearTerm}.
\end{proof}

\begin{rem}
In practice, it is enough to implement a split-step solver for 
the scalar field, $\psi\in\mathbb C$ 
rather than the vector field, $\boldsymbol\psi\in\mathbb R^2$, as was
done in~\cite{wang2013comparison}. The reason for providing the solution
operators, \eqref{eq:KerrSolution} and \eqref{eq:LpgKsolution}, 
in the vector case is that in the next subsection we will use them to derive 
solution operators for the  linearized equation.
\end{rem}

\subsection{The linearized solution operators}

Next, we state two propositions that give analytical formulae for the two solution operators for the linearized equation~\eqref{eq:FA linearized equation}
in the fiber amplifier. Setting $g=0$
gives the corresponding results for the single mode fiber segments.

\begin{proposition}\label{prop:LinKerrTerm}
The solution operator for the linearization,
\begin{equation}
\partial_t \mathbf u \,\,=\,\, 
\left[\mathcal M_1(\boldsymbol\psi) + \mathcal M_2(\boldsymbol\psi)\right]
\mathbf u,
\label{eq:LinKerrTerm}
\end{equation}
 of the Kerr nonlinearity term 
in the linearized equation~\eqref{eq:FA linearized equation} for the fiber amplifier is 
\begin{align}
\mathbf u (t+h,x) 
\,\,&=\,\,
\exp(  \int_{t}^{t+h} \left[\mathcal M_1(\boldsymbol\psi(s)) + \mathcal M_2(\boldsymbol\psi(s))\right] \, ds) 
\, \mathbf u (t,x)
\nonumber
\\
\,\,&=\,\,
\mathbf R( \gamma \| \boldsymbol\psi (t,x) \|^2 h)
\left( \mathbf I + 2 \gamma h \,\mathbf J \boldsymbol\psi(t) \boldsymbol\psi(t)^T\right)
\mathbf u(t,x).
\label{eq:LinKerrSolution}
\end{align}
\end{proposition}

\begin{proof}
Suppose that $\mathbf u$ solves \eqref{eq:LinKerrTerm}. Then 
$\boldsymbol \psi_\epsilon = \boldsymbol \psi + \epsilon \mathbf u$
solves \eqref{eq:KerrTerm} and so by \cref{prop:KerrTerm}, 
\begin{equation}
\boldsymbol \psi_\epsilon (t+h) \,\,=\,\, F(\epsilon) 
\boldsymbol \psi_\epsilon (t),
\quad\text{where } 
F(\epsilon) = \mathbf R(\theta(\epsilon)) 
\text{ with } 
\theta(\epsilon) = \gamma \| \boldsymbol \psi_\epsilon \|^2 h.
\label{eq:KerrTermFEqn}
\end{equation}
Keeping only those terms that are linear in $\epsilon$ gives
$\mathbf u(t+h) = F(0) \mathbf u(t) + F'(0)\boldsymbol\psi(t)$.
The result  now follows as $F'(\epsilon) = F(\epsilon) \mathbf J  \theta'(\epsilon) 
$ and $\theta'(\epsilon) = 2\gamma h \boldsymbol\psi(t)^T 
\mathbf u(t) + 2 \epsilon \| \mathbf u(t)\|^2$.
\end{proof}

\begin{proposition}\label{prop:LinHMELinearTerm}
The solution operator for the term,
\begin{equation}
\partial_t \mathbf u \,\,=\,\, 
\left[\mathcal L + g( \boldsymbol\psi)\mathcal K + \mathcal P( \boldsymbol\psi) \right] \mathbf u,
\label{eq:LinHMELinearTerm}
\end{equation}
in the linearized equation~\eqref{eq:FA linearized equation}  is given  in the Fourier domain by
\begin{equation}
\widehat{\mathbf u}(t+h/2,\omega) 
\,\,=\,\,
e^{a(\omega) G(t,t+h/2)} \, 
\mathbf R( b(\omega)h/2) 
\left[ a(\omega) \widehat{\boldsymbol\psi}(t,\omega) \, 
\frac{\partial G}{\partial u} \,\,+\,\, \widehat{\mathbf u}(t, \omega)\right],
\label{eq:LinHMELinearTermSoln}
\end{equation}
where the directional derivative of nonlocal gain is given by 
\begin{equation}
\frac{\partial G}{\partial u} 
\,\,=\,\,
\frac h2 \left( \frac{\partial g}{\partial u} +  \frac h4 \frac{\partial g_2}{\partial u} \right) \,\,+\,\, \mathcal O(h^3),
\label{eq:dGduapprox}
\end{equation}
with 
\begin{equation}
\frac{\partial g}{\partial u}  \,\,=\,\,
\frac{-2g^2(\boldsymbol\psi)}{g_0 E_\emph{sat}}
\,\langle \boldsymbol\psi, \mathbf u \rangle
\label{eq:dgdu}
\end{equation}
and
\begin{equation}
\frac{\partial g_2}{\partial u}  \,\,=\,\,
\frac{-2g(\boldsymbol\psi)}{g_0 E_\emph{sat}}
\left[
2g^2(\boldsymbol\psi)
\,\langle \mathcal K\boldsymbol\psi, \mathbf u \rangle\,\,+\,\,
( 3 g(\boldsymbol\psi) 
\,\langle \mathcal K\boldsymbol\psi, \boldsymbol\psi \rangle +
2 \langle \boldsymbol\psi, \mathcal L \boldsymbol\psi \rangle) \frac{\partial g}{\partial u}
\right].
\label{eq:dg2du}
\end{equation}
\end{proposition}

\begin{rem}
The inner products in \eqref{eq:dgdu} and \eqref{eq:dg2du}
are the $L^2$-inner products
$\langle \cdot, \cdot \rangle = 
\langle \cdot, \cdot \rangle_{L^2(\mathbb R, \mathbb C^2)}$.
These can be computed in the frequency domain using the formulae
\begin{align}
\langle \mathcal K\boldsymbol\psi, \mathbf u \rangle
\,\,&=\,\,
\int_{-\infty}^{\infty} a(\omega) \, \widehat{\boldsymbol\psi}^*(\omega)\,
\widehat{\mathbf u}(\omega)
\, d\omega,
\\
\langle \mathcal K\boldsymbol\psi, \boldsymbol\psi \rangle
\,\,&=\,\,
\int_{-\infty}^{\infty} a(\omega) \, \| \widehat{\boldsymbol\psi}(\omega)\|^2\, d\omega,
\\
\langle \boldsymbol\psi, \mathcal L \boldsymbol\psi \rangle
\,\,&=\,\,
\int_{-\infty}^{\infty} b(\omega) \, \widehat{\boldsymbol\psi}^*(\omega)\,
\mathbf J\,
\widehat{\boldsymbol\psi}(\omega)
\, d\omega.
\end{align}
\end{rem}

\begin{proof}
The proof is similar to that of \cref{prop:LinKerrTerm}.
Let $G(\boldsymbol\psi,h) = G(t,t+h)$ as in \eqref{eq:Gexact}.
This time we set
$F(\epsilon)(\omega) = e^{a(\omega) G(\epsilon,h)} \,
\mathbf R( b(\omega)h/2)$, where
$G(\epsilon,h) = G(\boldsymbol\psi + \epsilon\mathbf u,h)$.
Then $F'(\epsilon) = a(\omega) \partial_\epsilon G(\epsilon,h)\, F(\epsilon)$.
Therefore, \eqref{eq:LinHMELinearTermSoln} follows by defining
$\frac{\partial G}{\partial u} =  \partial_\epsilon G(0,h)$, in accordance with the
definition of the directional derivative. 

Next, \eqref{eq:dGduapprox} follows from \eqref{eq:Gapprox}
where 
\begin{equation}
\frac{dg}{du} \,\,:=\,\, \left.\frac{\partial}{\partial\epsilon}\right |_{\epsilon=0}g(\boldsymbol\psi+\epsilon\mathbf u)
\,\,=\,\, 
\frac{-g^2(\boldsymbol\psi)}{g_0E_{\text{sat}}} \frac{\partial E}{\partial u}
\,\,=\,\, 
\frac{-2g^2(\boldsymbol\psi)}{g_0E_{\text{sat}}} \langle \boldsymbol\psi,\mathbf u\rangle,
\end{equation}
and $\frac{dg_2}{du}$ is calculated as follows. First, 
as functions of $x$
we have that
\begin{equation}
g_2(\boldsymbol\psi) = F_1(\boldsymbol\psi)F_2(\boldsymbol\psi),
\label{eq:g2product}
\end{equation}
where 
\begin{equation}
F_1(\boldsymbol\psi) = \frac{-2g^2(\boldsymbol\psi)}{g_0E_{\text{sat}}}
\quad\text{ and }\quad
F_2(\boldsymbol\psi) = \langle \boldsymbol\psi, (\mathcal L + g(\boldsymbol\psi) \mathcal K) \boldsymbol\psi \rangle.
\end{equation}
Now 
\begin{equation}
\frac{\partial F_1}{\partial u} \,\,=\,\, \frac{-4g(\boldsymbol\psi)}{g_0E_{\text{sat}}}
\, \frac{\partial g}{\partial u}
\end{equation}
and
\begin{align}
\frac{\partial F_2}{\partial u} \,\,&=\,\, 
\langle \mathbf u, (\mathcal L + g(\boldsymbol\psi)\mathcal K)
\boldsymbol\psi \rangle
\,\,+\,\,
\langle \boldsymbol\psi, \mathcal K \boldsymbol\psi \rangle 
\frac{\partial g}{\partial u}
\,\,+\,\,
\langle  \boldsymbol\psi, (\mathcal L + g(\boldsymbol\psi)\mathcal K)
\mathbf u \rangle
\\
\,\,&=\,\,
2g(\boldsymbol\psi) \langle \mathcal K \boldsymbol\psi, \mathbf u\rangle
\,\,+\,\,
\langle \boldsymbol\psi, \mathcal K \boldsymbol\psi \rangle 
\frac{\partial g}{\partial u},
\end{align}
since $\mathcal K^* = \mathcal K$ and $\mathcal L^* = -\mathcal L$.
Equation~\eqref{eq:dg2du} now follows by applying the product rule to \eqref{eq:g2product}.
\end{proof}

\subsection{Solution operators for the adjoint linearized equations}

In this subsection, we describe the split-step method we used to solve the
adjoint linearized equation in the fiber amplifier, which is not completely
straightforward due to the nonlocal saturable gain, $g$.

Since the adjoint linearized equation in a fiber segment 
is solved backwards in time, we introduce the backwards time 
variable, $s=L-t$, where $L$ is the length of the segment. 
By \eqref{eq:Adjoint linearized IVP}, in the fiber amplifier the adjoint equation is given by
\begin{equation}
\partial_s \mathbf v \,\,=\,\, 
\left( \mathcal L^T + g(\boldsymbol\psi(t))\mathcal K^T 
+ [\mathcal M_1(\boldsymbol\psi(t))]^T 
+ \mathcal [M_2(\boldsymbol\psi(t))]^T
+ \mathcal [P(\boldsymbol\psi(t))]^T
\right) \, \mathbf v,
\label{eq:AdjointFA}
\end{equation}
where $\mathcal L$ and $\mathcal M_1$ are antisymmetric,
$\mathcal K$ is symmetric, and $\mathcal M_2^T = -2\gamma\boldsymbol\psi\boldsymbol\psi^T\mathbf J$.
Next, we recall from \eqref{eq:OperatorP} that  $\mathcal P(\mathbf u) = 
\frac{-2g^2}{g_0E_{\text{sat}}} \mathcal K\boldsymbol\psi\, \langle \boldsymbol\psi,\mathbf u\rangle$.
A calculation based on the defining formula for the 
adjoint~\eqref{eq:Inner product equality}
shows that
\begin{equation}
\mathcal P^T(\mathbf v) = 
\frac{-2g^2}{g_0E_{\text{sat}}} \boldsymbol\psi\, \langle \mathcal K\boldsymbol\psi , \mathbf v\rangle.
\end{equation}

\begin{proposition}
The solution operator for the adjoint equation~\eqref{eq:AdjointFA}, 
in a fiber amplifier of length, $L$, 
is given up to terms of order $\mathcal O(h^3)$ by
\begin{align}
\mathbf v(s) \,\,&=\,\, \mathcal U^*(s,s-h) \mathbf v(s-h)
\nonumber
\\
\,\,&=\,\,
[\exp( \mathcal A(t+h/2, t))]^*\,
[\exp( \mathcal B(t+h, t))  ]^* \,
[\exp( \mathcal A(t+h, t+h/2))]^*,
\label{eq:AdjointSolnOp}
\end{align}
where $t=L-s$, and the split solution operators are given by
\begin{equation}
[\exp( \mathcal B(t+h, t)) ]^* \,\,=\,\, (\mathbf I - 2\gamma \boldsymbol\psi(t)
\boldsymbol\psi(t)^T \mathbf J) \, \mathbf R(-\gamma \| \boldsymbol\psi(t) \|^2 h)
 \end{equation}
which is most readily computed 
in the fast-time domain, and
\begin{align}
[\exp( \mathcal A(t+h/2, t))]^*\mathbf v \,\,&=\,\,
\nabla G(h/2) \, \langle \mathcal K\boldsymbol\psi(t), 
\mathbf w \rangle +\mathbf w,
\label{eq:AdjointSolnA}
\\
\mathbf w \,\,&=\,\, 
\exp[-\mathcal L h/2 + G(\boldsymbol\psi(t),h/2)\mathcal K]\,\mathbf v,
\nonumber
\end{align}
which is most readily computed in the frequency domain. Here,
\begin{equation}
\nabla G(\boldsymbol\psi , h/2)  \,\,=\,\, \tfrac h2 \left[ \alpha_1 \boldsymbol\psi 
+ \tfrac h4 ( \alpha_2 \mathcal K \boldsymbol\psi  + \alpha_3 \boldsymbol\psi )\right],
\end{equation}
where
\begin{equation}
\alpha_1 = \frac{-2g^2}{g_0E_{\operatorname{sat}}},
\qquad
\alpha_2=2g\alpha_1,
\qquad
\alpha_3 = \frac{-2g}{g_0E_{\operatorname{sat}}}
\left( 3g\langle \mathcal K \boldsymbol\psi, \boldsymbol\psi \rangle
+ 2 \langle \boldsymbol\psi, \mathcal L \boldsymbol\psi \rangle \right) \alpha_1,
\end{equation}
are all evaluated at time, $t$.
\end{proposition}

\begin{proof}
We recall from \eqref{eq:SSS} that the solution operator
for the linearized equation from time $t$ to $t+h$ is  of the form
\begin{align}
\mathbf u(t+h) \,\,&=\,\, \mathcal U(t+h,t)\,\mathbf u(t)
\nonumber
\\
\,\,&=\,\, \exp( \mathcal A(t+h, t+h/2)) \exp( \mathcal B(t+h, t))  
\exp( \mathcal A(t+h/2, t)) \, \mathbf u(t),
\end{align}
where the operators $\mathcal A$ and $\mathcal B$ are given in \cref{prop:LinKerrTerm} and \cref{prop:LinHMELinearTerm}, respectively.
Since the forward time interval $[t,t+h]$ corresponds to the backward time interval
$[s-h,s]$,  the solution operator for the adjoint equation is given by
\begin{align}
\mathbf v(s) \,\,=\,\, \mathcal U^*(s,s-h) \mathbf v(s-h)
\,\,=\,\, [\mathcal U(t+h,t)]^* \mathbf v(s-h),
\end{align}
from which we obtain~\eqref{eq:AdjointSolnOp}. To establish~\eqref{eq:AdjointSolnA}, we first observe that
the gradient, $\nabla G$, is defined so that 
$\frac{\partial G}{\partial u} \,\,=\,\,\langle \nabla G,\mathbf u\rangle$.
Then, as in~\eqref{eq:LinHMELinearTermSoln}, 
\begin{equation*}
\mathbf u(t+h) \,\,=\,\, \exp( A(t+h/2,t))\, \mathbf u(t) 
 \,\,=\,\, 
 \exp(\mathcal L h/2 + G(\boldsymbol\psi,h/2)\mathcal K)\,
 (\mathbf u(t) + \mathcal K \boldsymbol\psi \langle \nabla G,\mathbf u\rangle ).
\end{equation*}
Equation~\eqref{eq:AdjointSolnA} now follows from the identity
$\langle \mathcal T(\mathbf f \langle \mathbf g,\mathbf u\rangle ), \mathbf v \rangle \,\,=\,\, 
\langle \mathbf u, \langle \mathbf f, \mathcal T^*\mathbf v \rangle \mathbf g \rangle$.
\end{proof}

\section{Spectrum of the Monodromy Operator}\label{sec:Spectrum}

In analogy with the Floquet theory of periodic solutions of nonlinear 
ordinary differential equations~\cite{teschl2012ordinary}, we expect that
the stability of  a periodically
stationary pulse solution, $\boldsymbol\psi$, of  a lumped laser model can be determined by
the spectrum, $\sigma(\mathcal M)$, of the monodromy operator. 
The spectrum of $\mathcal M$ is the union of the essential spectrum, 
$\sigma_{\text{ess}}(\mathcal M)$, and the 
eigenvalues~\cite{zweck2021essential}. 
In~\cite{shinglot2023essential,shinglot2022continuous}
we derived a formula for $\sigma_{\text{ess}}(\mathcal M)$. 
As in \cite{cuevas2017floquet,shinglot2022continuous},  we approximate $\sigma(\mathcal M)$
by the set of eigenvalues of a matrix approximation, $\mathbf M$,
 of the operator $\mathcal M: L^2(\mathbb R, \mathbb R^2) \to L^2(\mathbb R, \mathbb R^2)$. To do so, 
  we first truncate the domain, $\mathbb R$, to a finite interval,
  which we then discretize using $N$ equally-spaced 
points, $x_j$.  Then, any function, $\boldsymbol\psi  = (\psi_1, \psi_2)^T \in  L^2(\mathbb R, \mathbb R^2)$, is approximated by a vector 
$[ \psi_1(x_0), \psi_2(x_0), \cdots,  
\psi_1(x_{N-1}), \psi_2(x_{N-1}]^T \in \mathbb R^{2N}$. As a consequence,  
the operator $\mathcal M$ can be approximated by a linear transformation
$\mathbb M: \mathbb R^{2N} \to \mathbb R^{2N}$. To compute the matrix,
$\mathbf M$,
of $\mathbb M$ in the standard basis, we recall that for each
$k\in \{1,\cdots, 2N\}$, the $k$-th column of  $\mathbf M$ is given by the action
of $\mathbb M$ on the $k$-th standard basis vector, $\mathbf e_k \in 
\mathbb R^{2N}$. That is, using \eqref{eq:Linearization of the round trip operator},
the $k$-th column of  $\mathbf M$
is obtained by numerically solving the linearized equations given in 
\cref{sec:Linearization of the Round Trip Operator} for one round trip of the laser
with an initial condition given by $\mathbf e_k$.

In the remainder of this section we present some theoretical results
about the spectrum of $\mathcal M$.
The linear stability of a stationary pulse solution of the NLSE
is determined by the spectrum of the  linearized differential 
operator, $\mathcal L$. It is well known that $\mathcal L$  
has 
an eigenvalue with algebraic multiplicity four at $\lambda=0$, 
which is due to the phase and 
fast-time translation invariances of the NLSE~\cite{kaup1990perturbation}.  
In this section, we will show that a minor modification of the monodromy operator 
has a multiplicity two eigenvalue at $\lambda=1$. 
As in the case of the NLSE, these eigenvalues are due to the phase and time translation invariances of the lumped laser model.
In analogy with a result of  Haus and Mecozzi~\cite{JQE29p983}
for stationary pulses, we expect that perturbations which couple
into the corresponding eigenfunctions will result in shifts in the phase and position
of the pulse~\cite{JQE29p983}. A result of Lunardi for periodic
solutions of nonlinear parabolic equations~\cite{lunardi2012analytic} 
suggests that,
except for such phase and time shifts, a periodically stationary pulse solution of the lumped model
will behave stably if $\sup\{ | \lambda| \,: \, \lambda \in \sigma(\mathcal M), \lambda \neq 1 \} < 1$. However, we leave the precise formulation and proof of
such a result to a future paper.

We recall that a pulse, $\boldsymbol\psi$, is periodically stationary if $\mathcal R(\psi) = \mathbf R(\theta)\psi$
for some $\theta$, and that the optimization method
in \cref{sec:PeriodicPulse} computes the pair $(\boldsymbol\psi, \theta)$. 
Since Floquet theory only applies to  solutions that are actually periodic, 
we absorb the constant rotation, $\mathbf R(\theta)$, 
into $\mathcal R$ by defining a \emph{modified round trip operator} by
$\widetilde{\mathcal R} := \mathbf R(-\theta) \circ \mathcal R$, so that
$\widetilde{\mathcal R} \boldsymbol\psi =  \boldsymbol\psi$.
We also have a \emph{modified monodromy operator}, 
$\widetilde{\mathcal M} := \mathbf R(-\theta) \circ \mathcal M$.

\begin{proposition}\label{prop:phtr}
Let $\boldsymbol\psi$ be a periodically stationary pulse with 
${\mathcal R} \boldsymbol\psi =  \mathbf R(\theta)\boldsymbol\psi$,
 and 
suppose that $\boldsymbol\psi, \boldsymbol\psi_x \in L^2(\mathbb R, \mathbb R^2)$.
Let 
\begin{equation}
\mathbf u_{\operatorname{ph}} = \mathbf J \boldsymbol\psi
\end{equation}
be the $\pi/2$-rotation of $\boldsymbol\psi$ and let
\begin{equation}
\mathbf u_{\operatorname{tr}} =  \boldsymbol\psi_x
\end{equation}
be the $x$-derivative of $\boldsymbol\psi$. 
Let  $\widetilde{\mathcal M}= \mathbf R(-\theta) \circ \mathcal M$ be the modified monodromy operator. 
Then
\begin{equation}
\widetilde{\mathcal M}\,\mathbf u_{\operatorname{ph}}  = \mathbf u_{\operatorname{ph}} 
\qquad\text{and}\qquad
\widetilde{\mathcal M}\,\mathbf u_{\operatorname{tr}}  = \mathbf u_{\operatorname{tr}}.
\end{equation}
Consequently, $\lambda=1$ is an eigenvalue of $\widetilde{\mathcal M}$
with multiplicity
(at least) two.
\end{proposition}

\begin{rem}
We call $\mathbf u_{\operatorname{ph}}$ the phase invariance eigenfunction
and $\mathbf u_{\operatorname{tr}}$ is the translation invariance 
eigenfunction. We note that the $\mathcal M$ itself does not generically have 
any eigenvalues on the unit circle. 
\end{rem}

\begin{rem}
The NLSE has the soliton solution 
\begin{equation}
\psi(t,x) \,\,=\,\, A \sech\{ A[x-x_0 + \Omega t]\} \, \exp\{ i [ \Phi + \tfrac 12
(A^2 - \Omega^2)t - \Omega x]\}.
\end{equation}
Just as in \cref{prop:phtr}, 
the phase  and fast-time invariances of the NLSE give rise
to two eigenvalues at zero with eigenfunctions given by 
$\psi_\Phi$ and $\psi_{x_0}$, respectively. (Here $\psi_p$ denotes the partial
derivative of $\psi$ with respect to a parameter, $p$.) In addition, if
$\mathcal L$ denotes the linearized operator, then
$\mathcal L \psi_A = A \psi_\Phi$ and $\mathcal L \psi_{\Omega} = A \psi_{x_0}$,
which gives rise to two Jordan blocks, one  associated with 
$\{\psi_\Phi, \psi_A\}$  and the second with $\{\psi_{x_0}, \psi_\Omega\}$~\cite{kaup1990perturbation}.
Consequently,   $\lambda=0$ is an eigenvalue with algebraic multiplicity four.
From another perspective,  for the NLSE,
 $\mathcal L$ is a real Hamiltonian operator, which implies that if $\lambda$ is an eigenvalue then so are $-\lambda$ and 
$\pm \overline{\lambda}$~\cite{kevrekidis2009discrete}.  
However, in our situation, although the monodromy operator, 
$\mathcal M$, is real it is not Hamiltonian and the Jordan blocks
involving  the amplitude and frequency eigenfunctions do not exist.
Consequently, the
eigenvalue at $\lambda=1$ only has algebraic multiplicity two.  
Furthermore,  we recall that 
when one linearizes an autonomous ordinary differential equation
about a time periodic solution the resulting monodromy operator
always has an eigenvalue $\lambda=1$ due to the time-invariance of the
nonlinear equation~\cite{teschl2012ordinary}.
In the context of the  Kuznetsov-Ma breather solution of the NLSE
this corresponds
 to an additional pair of eigenvalues at $\lambda=1$~\cite{cuevas2017floquet}.
However, this phenomenon does not occur in our context as the lumped model we are studying is not autonomous.
\end{rem}

\begin{proof}
First,  let $\boldsymbol\psi_\epsilon$ be the perturbation of 
$\boldsymbol\psi$ given by the phase rotation 
$\boldsymbol\psi_\epsilon =  \mathbf R(\epsilon) \boldsymbol\psi_0$, and let
$\mathbf u := \lim\limits_{\epsilon\to 0} \frac{\boldsymbol\psi_\epsilon -
\boldsymbol\psi}{\epsilon}$.  Then 
$\mathbf u =  \mathbf R'(0) 
\boldsymbol\psi = \mathbf J \boldsymbol\psi_0$ is a $\pi/2$-rotation
of $\boldsymbol\psi$. On the other hand, by the phase-shift invariance of each of
the  nonlinear operators, $\mathcal P$, we have that
\begin{equation}
\widetilde{\mathcal M}(\mathbf u) 
\,\,=\,\,  
\lim\limits_{\epsilon\to 0} \frac{\widetilde{\mathcal R}(\boldsymbol\psi_\epsilon) -
\widetilde{\mathcal R}(\boldsymbol\psi)}{\epsilon} 
\,\,=\,\,  
\lim\limits_{\epsilon\to 0} \frac{\boldsymbol\psi_\epsilon - \boldsymbol\psi}{\epsilon}
\,\,=\,\,  
\mathbf u.
\label{eq:MuRu}
\end{equation}
If instead we let   $\psi_\epsilon$ be the time translation of $\boldsymbol\psi$
given by $\boldsymbol\psi_\epsilon(x) = \boldsymbol\psi(x+\epsilon)$, then
$\mathbf u = \boldsymbol\psi_x$ is the $x$-derivative
of $\boldsymbol\psi$ and because of the fast-time translation invariance
of all the operators, $\mathcal P$, we again obtain~\eqref{eq:MuRu}.
\end{proof}

Since $\widetilde{\mathcal M} : L^2(\mathbb{R}, \mathbb{R}^2) \to L^2(\mathbb{R}, \mathbb{R}^2)$ is a real operator, the elements of the spectrum
are either real or come in complex conjugate pairs. 
In~\cite{shinglot2023essential}, we proved that under reasonable
assumptions on the system parameters and on the smoothness and decay  of the pulse,  the essential spectrum is given by 
\begin{equation}
    \sigma_{\operatorname{ess}}(\widetilde{\mathcal{M}}) \,\,=\,\, \{\, \lambda_{\pm}(\omega) \in \mathbb{C} \,\, | \,\, \omega \in \mathbb{R} \,\} \cup \{ 0 \},
    \label{eq:EssSpec1}
\end{equation}
where
\begin{equation}
    \lambda_{\pm}(\omega) = \ell_{\operatorname{OC}}
    (1-\ell_0) \exp\left\{ \frac 12 \left( 1 - \frac{\omega^2}{\Omega_g^2} \right)\int_0^{L_{\operatorname{FA}}} g(\psi(t))dt \right\} \exp\left\{ \pm i
    (\tfrac{\beta_{\text{RT}}}2  \omega^2 - \theta) \right\}.
    \label{eq:EssSpec2}
\end{equation}
Here, 
$\beta_{\text{RT}} = \beta_{\text{SMF1}}\text{L}_{\text{SMF1}} + \beta_{\text{FA}}\text{L}_{\text{FA}} + \beta_{\text{SMF2}}\text{L}_{\text{SMF2}} + \beta_{\text{DCF}}$ is the round trip dispersion. Geometrically, 
$ \sigma_{\operatorname{ess}}(\mathcal{M})$ is a pair of counter-rotating 
spirals which have a Gaussian decay  in the radial direction. 
In~\cite{shinglot2023essential,shinglot2022continuous}, we discuss conditions which guarantee that the essential
spectrum is stable.

 \section{Simulation Results}\label{sec:SimResults}
 
For the simulation results we present here, we choose the parameters 
in the model to be similar to those in the experimental stretched pulse laser of Kim~\cite{kim2014sub}. The saturable absorber is modeled by 
\eqref{eq:Saturable absorber} with $\ell_0 = 0.2$
and $P_\text{sat} = 50$~W. 
The saturable absorber is followed by a segment of single mode fiber, SMF1, 
modeled by \eqref{eq:SMF_JZ}, with $\gamma=2\times10^{-3}$~(Wm)$^{-1}$,
$\beta_{\operatorname{SMF1}}=10$~kfs$^2$/m, (1~kfs$^2$ = $10^{-27}$~s$^2$), and $L_{\operatorname{SMF1}}=0.32$~m,
a fiber amplifier,  modeled by \eqref{eq:fiber amplifier},  with 
$g_0=6$m$^{-1}$, $E_{\text{sat}}=200$~pJ, $\Omega_g=50$~THz, $\gamma=4.4\times10^{-3}$~(Wm)$^{-1}$,
$\beta_{\operatorname{FA}}=25$~kfs$^2$/m, and $L_{\operatorname{FA}}=0.22$~m, and
a second segment of single mode fiber, SMF2, with the same parameters as SMF1, but with $L_{\operatorname{SMF2}}=0.11$~m.
The dispersion, $\beta_{\operatorname{DCF}}$, of the dispersion compensation element is chosen so that the round trip dispersion is $\beta_{\text{RT}}=-1~$kfs$^2$.
Finally the 50\% output coupler is modeled by \eqref{eq:Output coupler} 
with $\ell_{\operatorname{OC}}=\sqrt{0.5}$.
Unless otherwise stated, we used a time window $-L_X/2 \leq x \leq L_X/2$
of size $L_X = 10$~ps discretized with $N=512$ points. 

The algorithms were implemented in MATLAB. We used the 
quasi-Newton BFGS algorithm~\cite{wright2006numerical} as implemented in the function \textsf{fminunc} to find the optimal pulse. 
In particular, the optimization algorithm is provided with the gradient
of the objective function,  computed using the adjoint state method
described in \cref{thm:GradientPoincareMap}.
The computational time to perform the optimization and compute 
the monodromy matrix, $\mathbf M$,  and its spectrum on a 3.5~GHz
Macbook Pro is about 3 minutes. The computation of $\mathbf M$ was done in parallel using 12 processors.

We begin by discussing the accuracy of the numerical solvers for the 
roundtrip operator, $\mathcal R$, and the linearization, $\mathcal M$, of $\mathcal R$.
For these results we use two error measures, the absolute error
\begin{equation}
\mathcal E_{\text{abs}}(\boldsymbol\psi_{\text{approx}},\boldsymbol\psi_{\text{exact}})  \,\,=\,\,
 \left [ \int \| \boldsymbol\psi_{\text{approx}}(x) - \boldsymbol\psi_{\text{exact}} (x)\|_{\mathbb R^2}^2 \, dx\right]^{1/2},
\end{equation}
and the relative error
\begin{equation}
\mathcal E_{\text{rel}}(\boldsymbol\psi_{\text{approx}},\boldsymbol\psi_{\text{exact}})
\,\,=\,\,
\frac{\mathcal E_{\text{abs}}(\boldsymbol\psi_{\text{approx}},\boldsymbol\psi_{\text{exact}})}{
E(\boldsymbol\psi_{\text{exact}})^{1/2}},
\end{equation}
where the pulse energy, $E(\boldsymbol\psi_{\text{exact}})$, is 
given by \eqref{eq:PulseEnergy}.

For this study we used an initial pulse, $\boldsymbol\psi_0$, 
obtained by propagating a Gaussian pulse for ten round trips
of the system. The Gaussian was given by
\begin{equation}
g(x) \,\,=\,\,
\sqrt{P_0} \exp( -(x/\sigma)^2)
\label{eq:Gaussian}
\end{equation}
where $\sigma = \text{FWHM}/2 \sqrt{\log 2}$. By choosing
a peak power of $P_0 = 400$~W
and a full width at half maximum of $\text{FWHM}=300$~fs,
we obtained a reasonable approximation, $\boldsymbol\psi_0$,
to a periodically stationary pulse. 

To assess the accuracy of the numerical solver for the roundtrip operator, $\mathcal R$,
we first computed an exact solution by propagating the initial pulse,
$\boldsymbol\psi_0$, for one roundtrip of the system with a step-size of 
$\Delta t = 10^{-4}$. We then computed approximate solutions 
using step sizes of $\Delta t = 10^{-2}$, $5\times 10^{-3}$, $2\times 10^{-3}$, 
$10^{-3}$, $5\times 10^{-4}$, and $2\times 10^{-4}$, and computed the error between the approximate and exact solutions.
In the left panel of \cref{fig:ErrorVsDeltaz}, we plot the absolute error
in units of J$^{1/2}$ as a function of $\Delta t$. The portion of the curve
with $\Delta t \geq 10^{-3}$ has  a slope of 
$4.02$ as expected for the globally fourth-order method we used.
The floor below an error level of $10^{-16}$ is due to  round-off error,
primarily of the Fourier transform.
The relative error is  approximately $10^5$ times larger than the absolute error. 
So, for example,  $\mathcal E_{\text{rel}}=2.9\times 10^{-8}$ when $\Delta t = 10^{-2}$.

In the center panel of \cref{fig:ErrorVsDeltaz}, we show the corresponding results for the linearized operator,
$\mathcal M$. For each choice of time step, we linearized $\mathcal R$ 
about the pulse obtained
by propagating  $\boldsymbol\psi_0$ with a step-size of $\Delta t$, and 
we chose the initial pulse for the linearized operator to be the phase invariance eigenfunction, $\mathbf u_0 = i\boldsymbol\psi_0$. 
In this case, the plot also has a slope of $4.02$ where 
$\Delta t \geq 10^{-3}$, and 
$\mathcal E_{\text{rel}}=2.9\times 10^{-8}$ when $\Delta t = 10^{-2}$.
For the  translation invariance eigenfunction,
$\mathbf u_0 = \Delta_x \boldsymbol\psi_0$,
for $\Delta t \geq 10^{-3}$, the slope (not shown)  is the same as for the 
phase invariance eigenfunction but the absolute errors are about twice as large.

In the right panel of \cref{fig:ErrorVsDeltaz}, we show the corresponding results for the adjoint of linearized operator,
$\mathcal M^*$. 
For these results, we chose the initial pulse to be $\mathbf v_0 = 
\mathcal R(\boldsymbol\psi_0) - e^{i\theta} \boldsymbol\psi_0$,
where $\boldsymbol\psi_0$ is computed with $\Delta t = 10^{-4}$
and $\theta$ is the angle between $\boldsymbol\psi_0$ and
$\mathcal R(\boldsymbol\psi_0)$. We note that $\max | \boldsymbol\psi_0 | 
= 16.2$ and $\max | \mathbf v_0 | = 1.2$.
Once again, the plot has a slope of $4.02$ where 
$\Delta t \geq 10^{-3}$, and 
$\mathcal E_{\text{rel}}=8.7\times 10^{-8}$ when $\Delta t = 10^{-2}$.

\begin{figure}[H]
    \centering
    \includegraphics[width=0.32\linewidth]{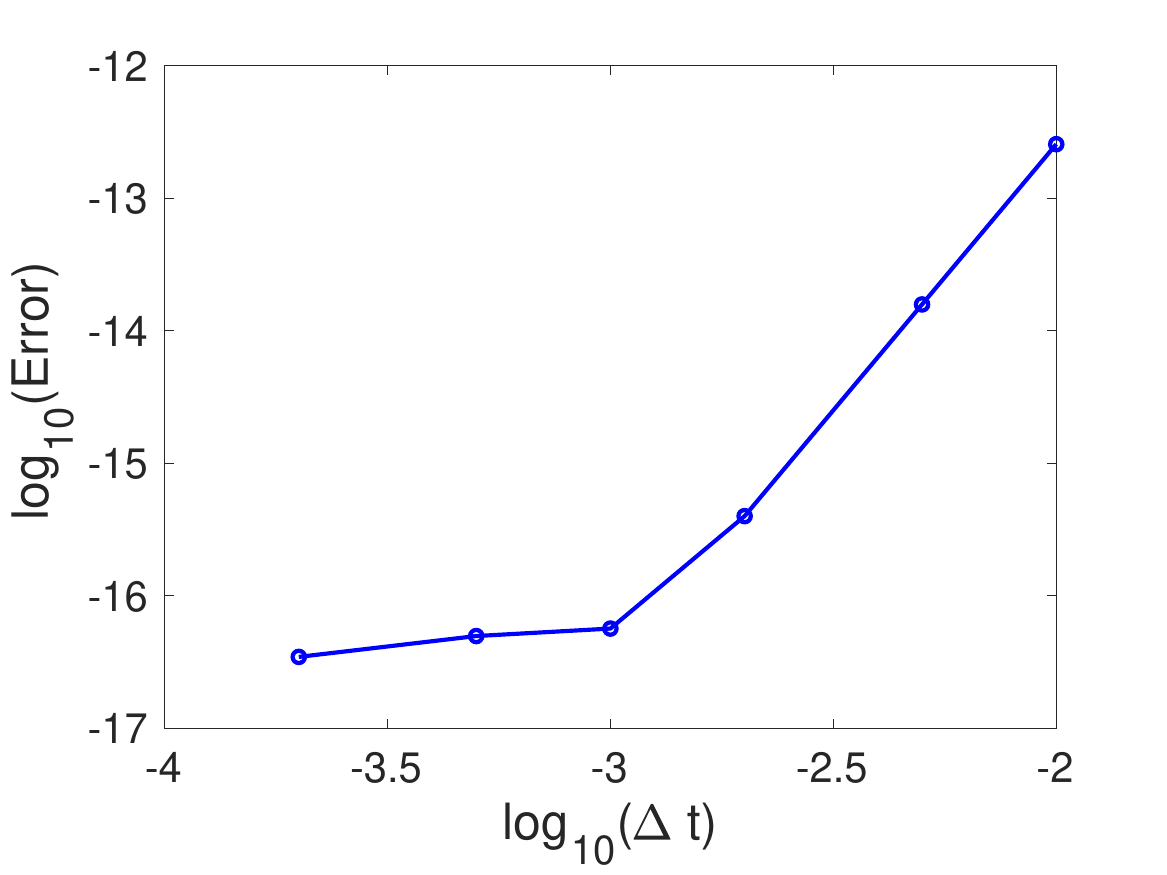}
    %\hskip40pt
    \includegraphics[width=0.32\linewidth]{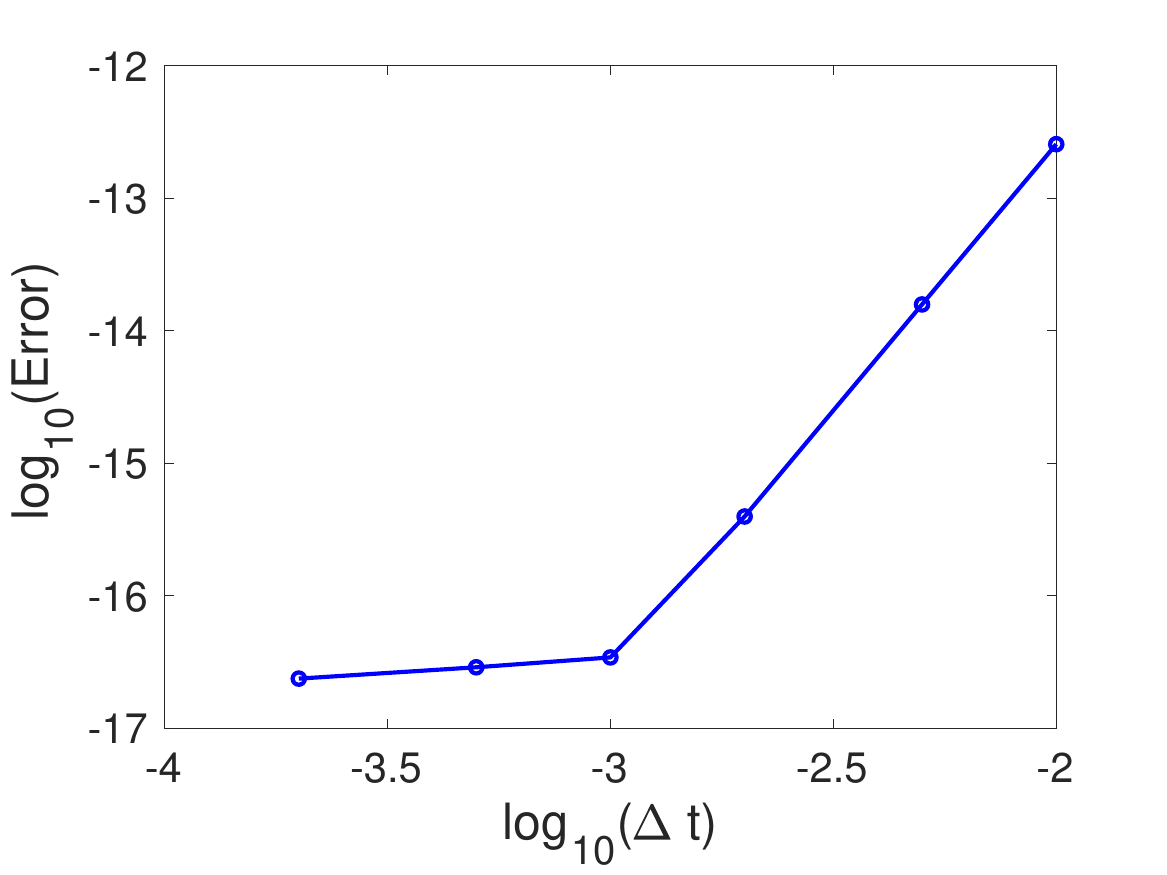}
    %\hskip40pt
    \includegraphics[width=0.32\linewidth]{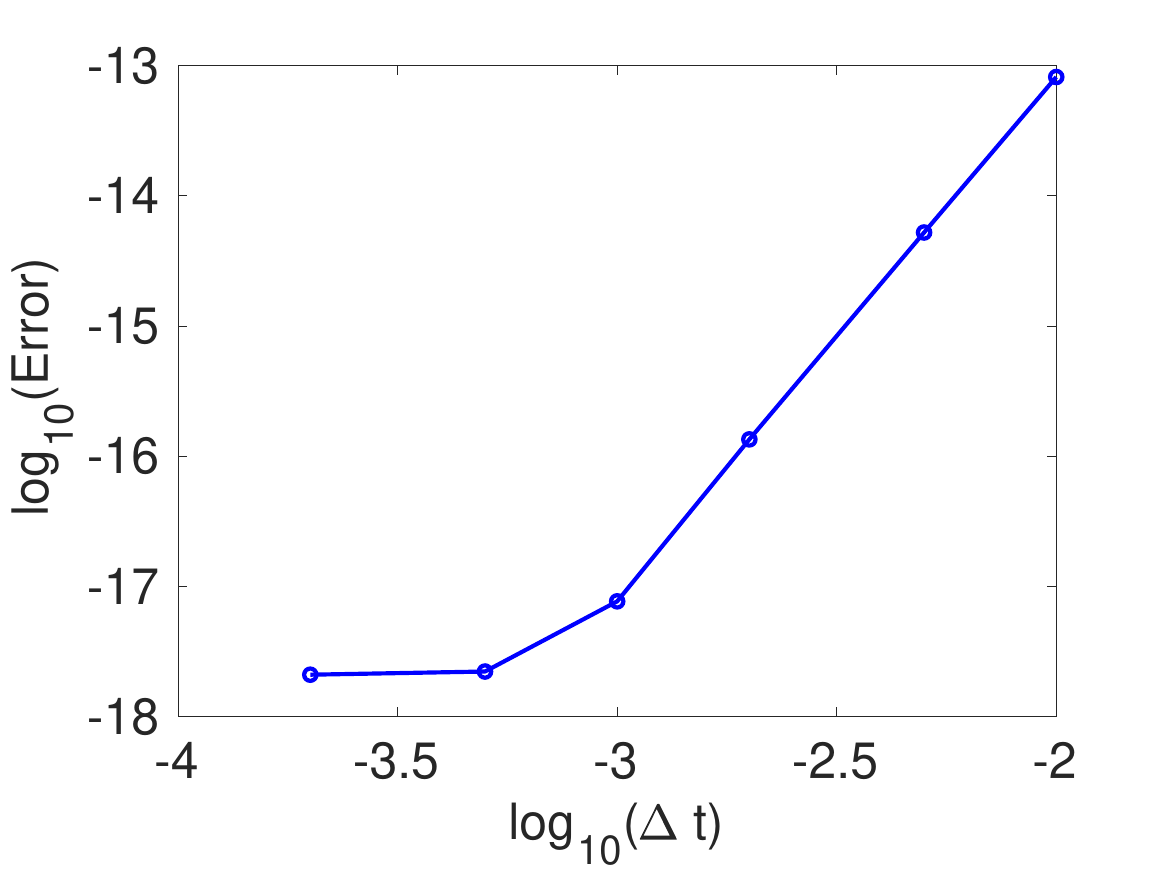}    
    \caption{Absolute error between an exact solution 
    (as computed with $\Delta t = 10^{-4}$) and
  approximate solutions with step size $\Delta t$
    for propagation over one round trip of the laser.
    Left: Result for the roundtrip operator, $\mathcal R$.
    Center: Result for the linearization, $\mathcal M$, of $\mathcal R$.
    Right: Result for the adjoint linearization, $\mathcal M^*$.
    }
    \label{fig:ErrorVsDeltaz}
\end{figure}

Even though the linearized round trip solver has the correct order of accuracy it is
nevertheless possible that the solution is not correct. To verify that 
the linearized round trip operator has been correctly derived and implemented
we must verify that
\begin{equation}
\mathcal M(\mathbf u_0) \,\,=\,\, \lim\limits_{\epsilon\to 0}
\frac{\mathcal R(\boldsymbol\psi_0 + \epsilon \mathbf u_0) - \mathcal R(\boldsymbol\psi_0)}{\epsilon} \,\,=:\,\, D_{\boldsymbol\psi_0}\mathcal R
(\mathbf u_0).
\label{eq:MviaRFinDiff}
\end{equation}
If we let $f(\epsilon) := \mathcal R(\boldsymbol\psi_0+\epsilon \mathbf u_0)(x) : \mathbb R \to\mathbb R^2$, then the directional derivative is given by 
$D_{\boldsymbol\psi_0}\mathcal R(\mathbf u_0)(x) = f'(0)$.
Due to round-off errors, standard finite difference approximations of $f'(0)$ 
are not accurate when $\epsilon$ is small. 
A commonly employed method is to use a complex step derivative approximation~\cite{lyness1967numerical}, 
which requires that $f$ is real valued. However 
this is not actually the case for the numerically computed 
$f$ because of small imaginary round-off errors in the computation
of the discrete Fourier transforms. Instead, we use a
spectral differentiation method of Fornberg~\cite{fornberg1981numerical}.
With this method, Cauchy's integral formula is applied to show that if $f:\mathbb C\to\mathbb C$ is  complex analytic
in a disc of radius  $R$ about a point $z_0\in \mathbb C$ then, for any $r\in [0,R]$,
\begin{equation}
f'(z_0) = \frac{1}{2\pi r}\int_0^{2\pi} F(t) e^{-it}\, dt,
\end{equation}
where $F(t) = f(z_0+re^{it})$.
Then $f'(z_0) = \frac{c_1}r$ where $c_1$ is the first Fourier coefficient in the Fourier series of $F$.
Using a discrete Fourier transform approximation with $M$ points, we find that
\begin{equation}
f'(z_0) \,\,\approx\,\, \frac 1{rM} \sum\limits_{m=0}^{M-1} F_m w^{-m},
\end{equation}  
where $w= e^{i2\pi/M}$ and   $F_m = F(w^m)$.

\begin{figure}[!tbh]
    \centering
    \includegraphics[width=0.3\linewidth]{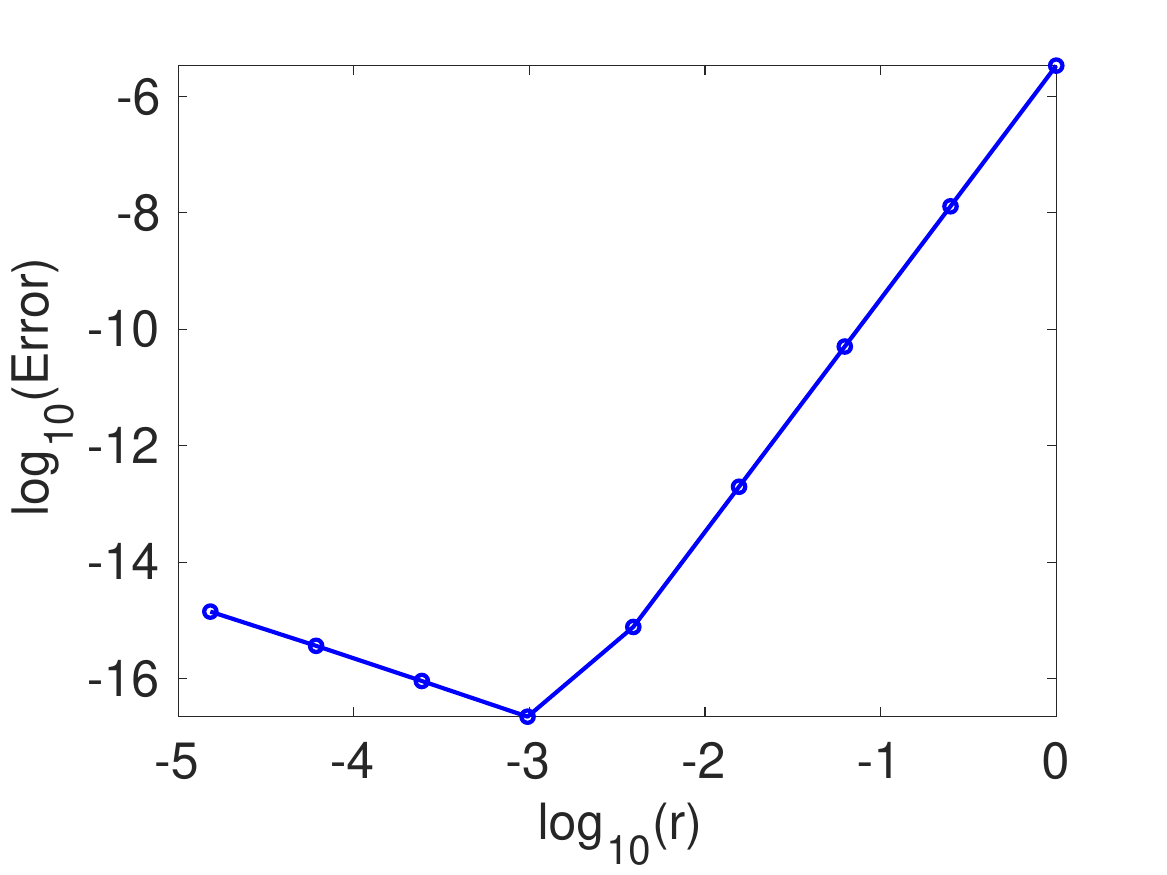}
    \hskip20pt
    \includegraphics[width=0.3\linewidth]{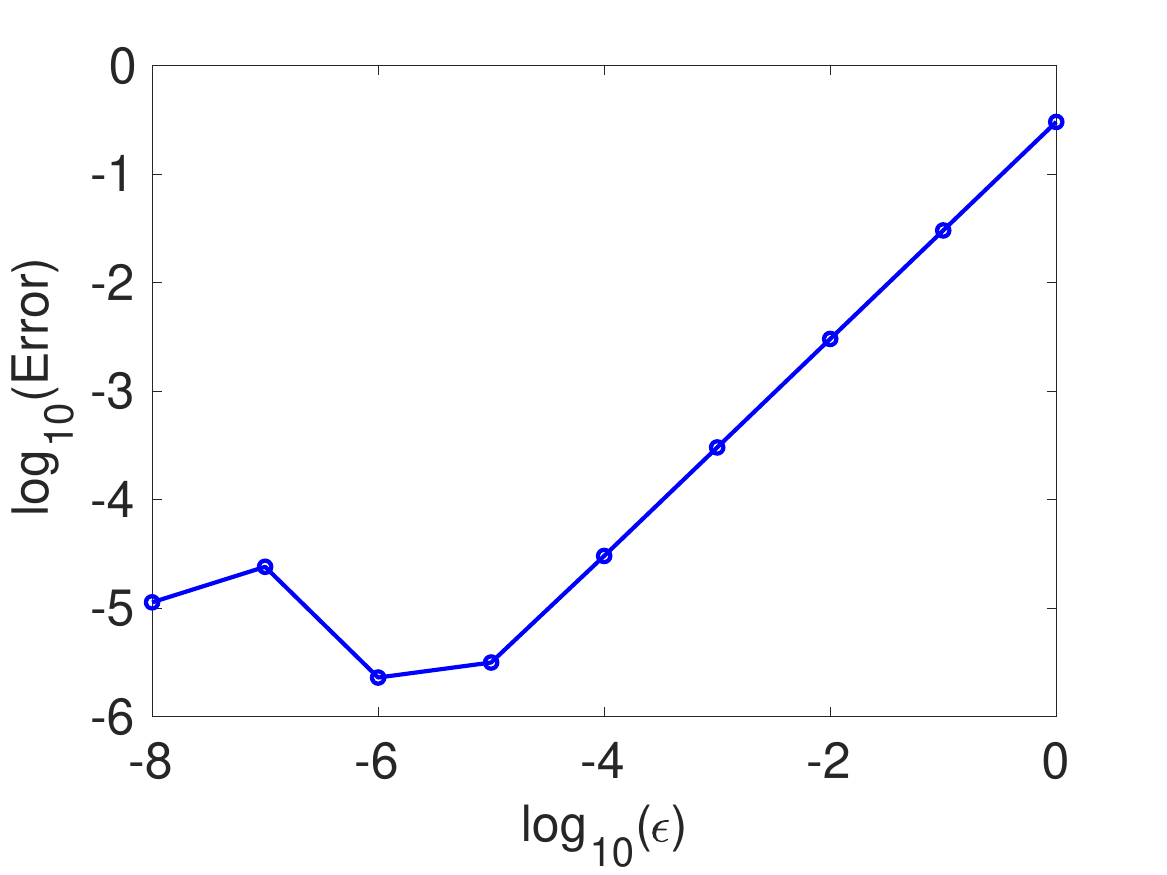}
    \caption{Left: Absolute error between the numerical solution of the linearized operator, $\mathcal M(\mathbf u_0)$,
     and the spectral  approximation  of the directional derivative, 
     $D_{\boldsymbol\psi_0}\mathcal R(\mathbf u_0)$, for the theoretical phase-shift 
     eigenfunction, $\mathbf u_0 = \mathbf J\boldsymbol\psi_0$.
     Right: Relative error between the
directional derivative of $\mathcal E$ computed in terms of the 
adjoint of $\mathcal M$ via \eqref{eq:Variational derivative of Poincare map} and \eqref{eq:Gradient of Poincare map wrt pulse}
and using a finite difference with increment, $\epsilon$.}
    \label{fig:SDError}
\end{figure}

To verify \eqref{eq:MviaRFinDiff},  we first extended 
$ f$ to a vector-valued complex analytic function $ f: \mathbb C \to\mathbb C^2$~\cite{martins2003complex}. 
To minimize the truncation error in the discretization of the Fourier series
we need $M$ to be sufficiently large. For the results presented here it was sufficient to choose $M=4$. Furthermore, to avoid round-off error
in the computation of the $F_m$ we do not want  $r$ to be too small~\cite{Boisgerault}.
 In the left panel of \cref{fig:SDError}, for the phase invariance eigenfunction,
 $\mathbf u_0 = i\boldsymbol\psi_0$,
 we plot the absolute error 
between $\mathcal M(\mathbf u_0)$ and the spectral derivative approximation of $f'(0)$ as a function of $r$. The minimum error is $2.4\times10^{-17}$ at  $r=2^{-10}$. 
Similar results were obtained for the translation invariance eigenfunction,
$\mathbf u_0 = \Delta_x\boldsymbol\psi_0$. These results were obtained using a time step of $\Delta t = 10^{-2}$. Decreasing the time step to $\Delta t = 10^{-3}$
we obtained a similar plot, except that the minimum error increased to 
$7.8\times10^{-17}$,  likely due to the larger accumulation of roundoff errors
in the numerical solution of the system model.

As a second test of the adjoint solver we examine the accuracy of 
the computation of the directional derivative, $D_{\boldsymbol\psi_0}
\mathcal E (\mathbf u_0)
= \langle \frac{\delta\mathcal E}{\delta \boldsymbol\psi_0}\, , \, 
\mathbf u_0 \rangle$, where 
the variational derivative, $\frac{\delta\mathcal E}{\delta \boldsymbol\psi_0}$, 
is given in terms of the
adjoint of $\mathcal M$ by \eqref{eq:Gradient of Poincare map wrt pulse}. For simplicity,  for this verification we approximate the directional derivative using a finite difference. So that roundoff errors do  not dominate, 
we need to ensure that the directional derivative is nonzero.
To ensure the variational derivative is not too close to zero we 
choose $\boldsymbol\psi_0$ to be a Gaussian with FWHM = 50~fs and $P_0 = 200$~W, which is not a periodically stationary pulse. 
In addition, we choose $\mathbf u_0$ so that the $L^2$-inner product is not zero. 
In the right panel of \cref{fig:SDError}, we show the relative error between the
directional derivative computed using a finite difference with increment, $\epsilon$, and the computation based on the adjoint of $\mathcal M$. 
For $\epsilon > 10^{-5}$, the slope of the error plot is 0.997, as expected
for a standard finite difference, which provides strong evidence for the
accuracy of the implementation  of the gradient of $\mathcal E$.

\begin{figure}[t]
    \centering
    \includegraphics[width=0.49\linewidth]{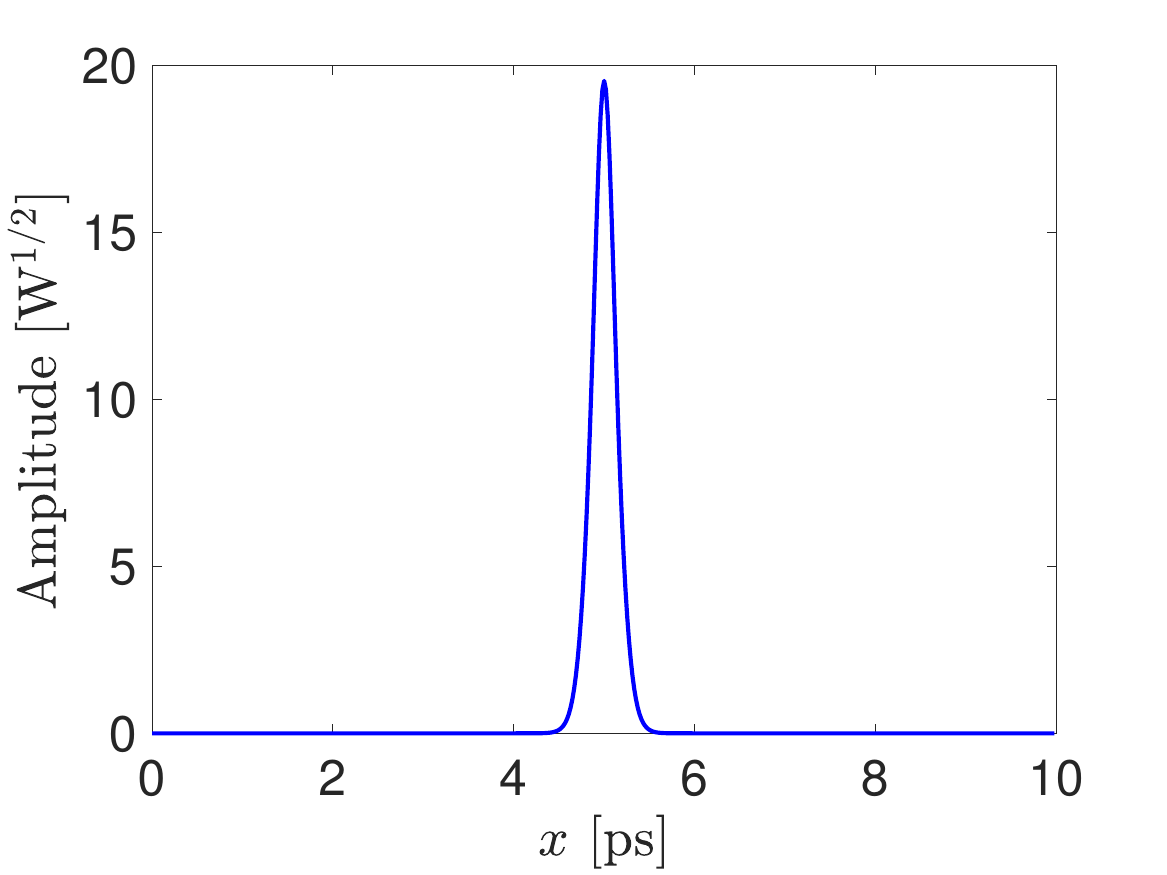}
    \includegraphics[width=0.5\linewidth]{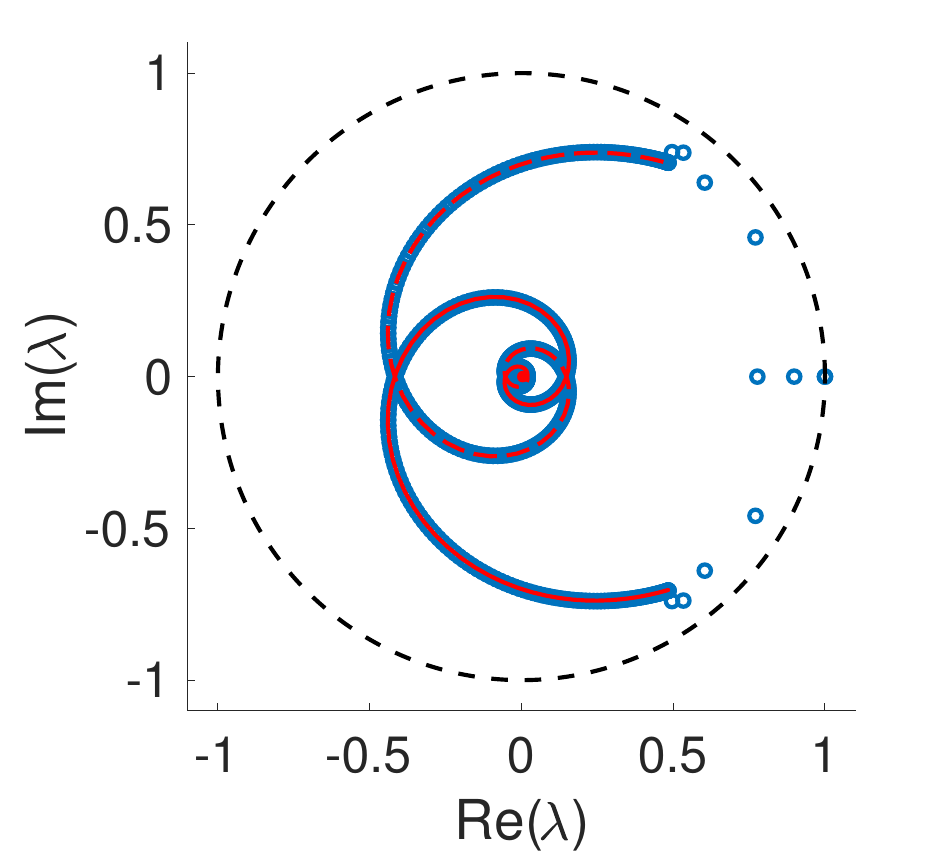}
      \caption{Left: Optimal pulse for the 
    periodically stationary pulse obtained using the parameters given 
    at the beginning of \cref{sec:SimResults}.
    Right:
    Spectrum of the monodromy operator for the optimal pulse shown in the left panel.  The eigenvalues of the discretized 
    operator are shown with blue circles and the essential spectrum obtained
    using \eqref{eq:EssSpec2} is shown with the solid red line.  
    }
    \label{fig:Spectrum}
\end{figure}

In  \cref{fig:Spectrum} (left panel), we show the instantaneous power
of the optimal pulse after the output coupler (also see \cref{fig:Short pulse laser model}, right, for a plot showing the evolution of this pulse through the laser).
The initial pulse for the optimization was obtained by evolving the Gaussian
~\eqref{eq:Gaussian} for
ten round trips at which point the value of the objective function in \eqref{eq:PoinObjFun}
was $\widetilde{\mathcal E}=8\times 10^{-3}$, which the optimization method then reduced 
to $\widetilde{\mathcal E}=5\times 10^{-27}$ in 36 iterations.
 In \cref{fig:Spectrum} (right panel), we show the numerically
computed  spectrum of the modified monodromy operator, 
$\widetilde{\mathcal M}$, with blue circles.  
A portion of this spectrum agrees with the essential spectrum obtained using \eqref{eq:EssSpec2}, which is shown with the solid red line. 
In addition, counting multiplicities, there are 12 eigenvalues that are not part of the essential spectrum. We label them, $\lambda_1, \cdots, \lambda_{12}$,
in order of decreasing magnitude.
First, there
is a multiplicity two eigenvalue at $\lambda=1$, which agrees with the
theoretical predictions in \cref{prop:phtr}.
The error in the phase invariance eigenvalue is  $10^{-13}$,
while that in the translation invariance eigenvalue is  $4\times10^{-11}$.
In \cref{fig:PhaseShiftEigenfunctions} we plot the amplitude,
$A(x) :=\| \mathbf u(x)\|_{\mathbb R^2}$,  of the 
corresponding 
phase invariance eigenfunction (left panel) and translation invariance 
eigenfunction (right panel). The numerically computed eigenfunctions
are shown with the blue dots and the (normalized) theoretical eigenfunctions
in \cref{prop:phtr} are shown with black solid lines. 
 The excellent agreement with both the essential spectrum and the theoretically predicted eigenvalues and eigenfunctions at $\lambda=1$ provides strong validation of the  numerical method.

There are two additional eigenvalues on the real axis at
$\lambda_5=0.8987$ and  $\lambda_{12}=0.7773$. 
The amplitude of the eigenfunction corresponding to
 $\lambda_5$, which  is shown with the red-dashed line in 
 the right panel of \cref{fig:PhaseShiftEigenfunctions},
 is very similar to the translation invariance eigenfunction.
 Similarly, the eigenfunction corresponding to $\lambda_{12}$,
 which  is shown with the red-dashed line in 
 the left panel of \cref{fig:PhaseShiftEigenfunctions},
 is very similar to the phase invariance eigenfunction.
 Finally, there are four eigenvalues near the edge of the upper arm
 of the essential spectrum. The corresponding eigenfunctions 
 are shown in \cref{fig:EigenfunctionsCloseFar}. We observe that the number of oscillations in the amplitude of these 
 eigenfunctions increases as the distance from the eigenvalue 
to the edge of the essential spectrum decreases.
 
 \begin{figure}[!t]
    \centering
    \includegraphics[width=0.49\linewidth]{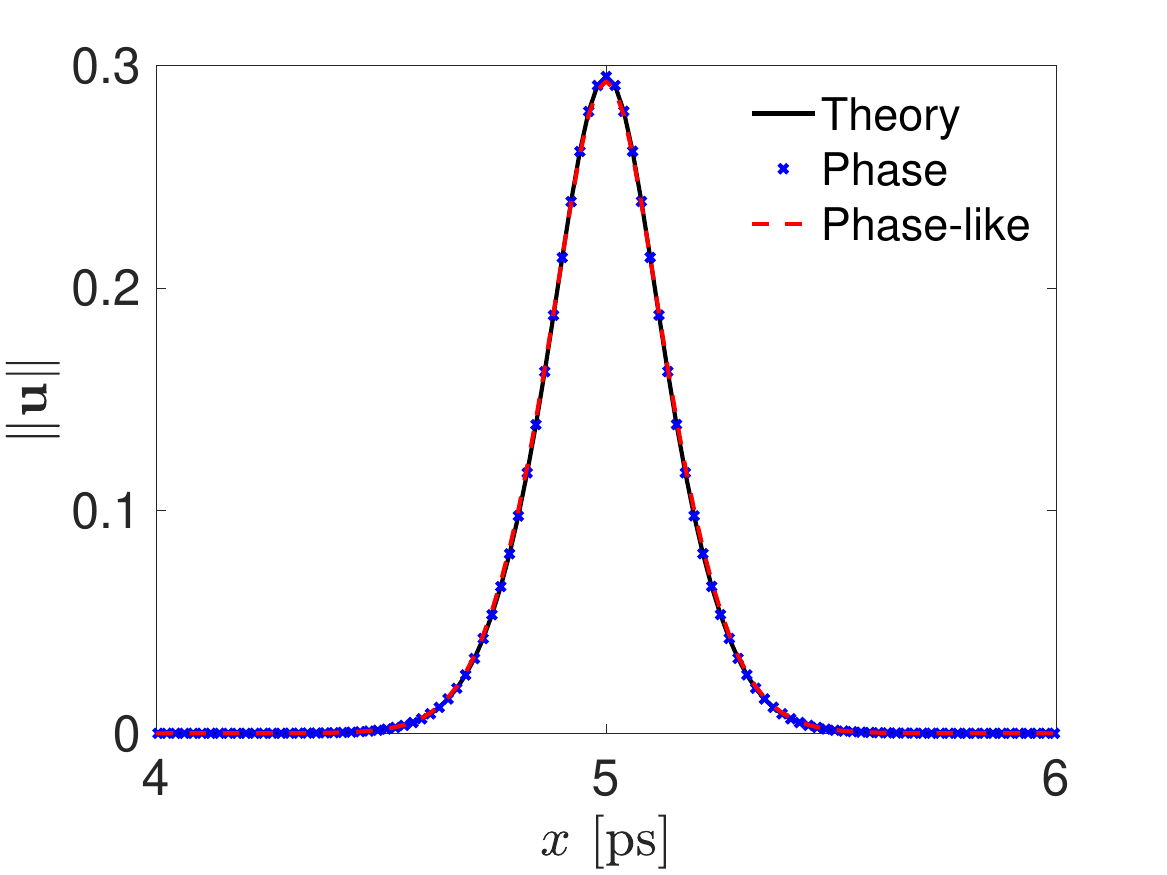}
   % \hskip20pt
    \includegraphics[width=0.49\linewidth]{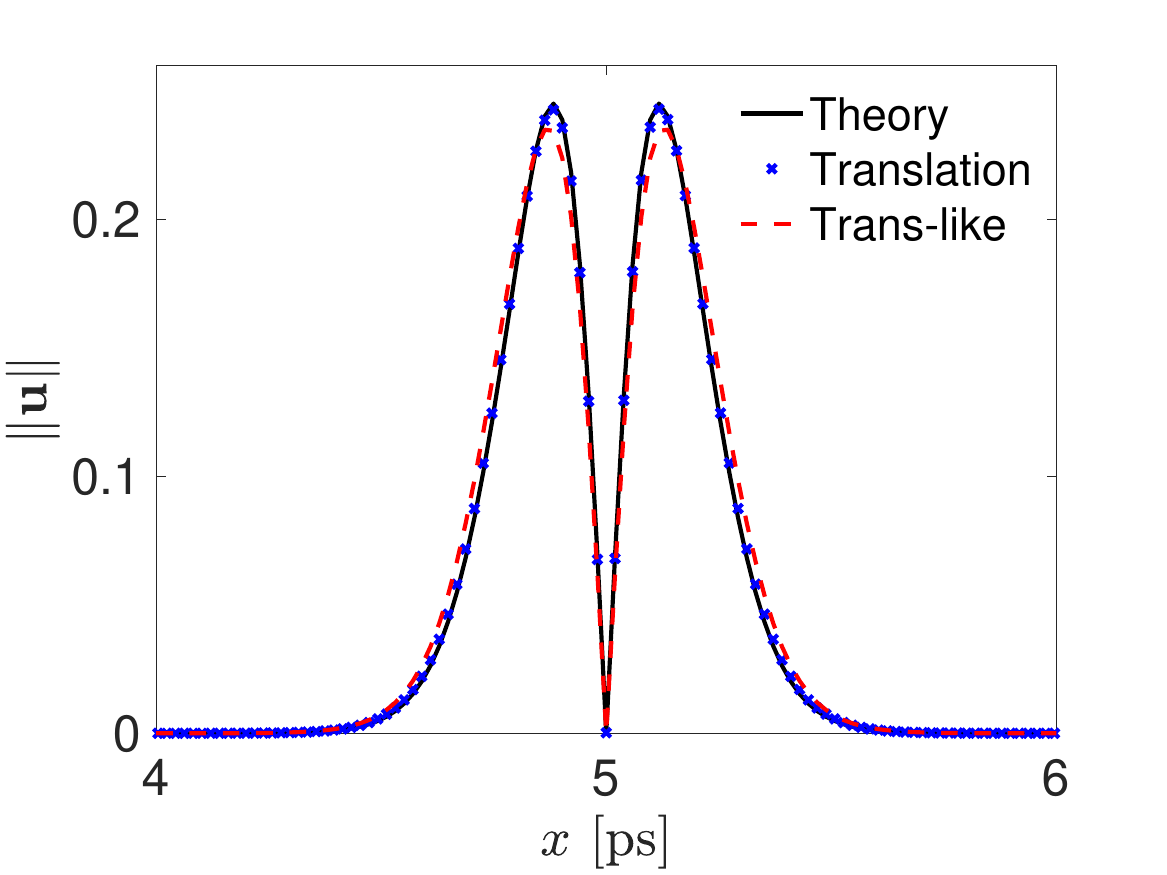}
    \caption{Left: Phase invariance eigenfunctions:
    theoretical  (black solid line) and numerical (blue dots) eigenfunctions
    with $\lambda=1$, and numerical eigenfunction corresponding to 
    $\lambda=0.7773$ (red dashed line).
     Right: Translation invariance eigenfunctions:
    theoretical  (black solid line) and numerical (blue dots) eigenfunctions
    with $\lambda=1$, and numerical eigenfunction corresponding to 
    $\lambda=0.8987$ (red dashed line).
    }
    \label{fig:PhaseShiftEigenfunctions}
\end{figure}
 
 \begin{figure}[!tbh]
    \centering
        \includegraphics[width=0.49\linewidth]{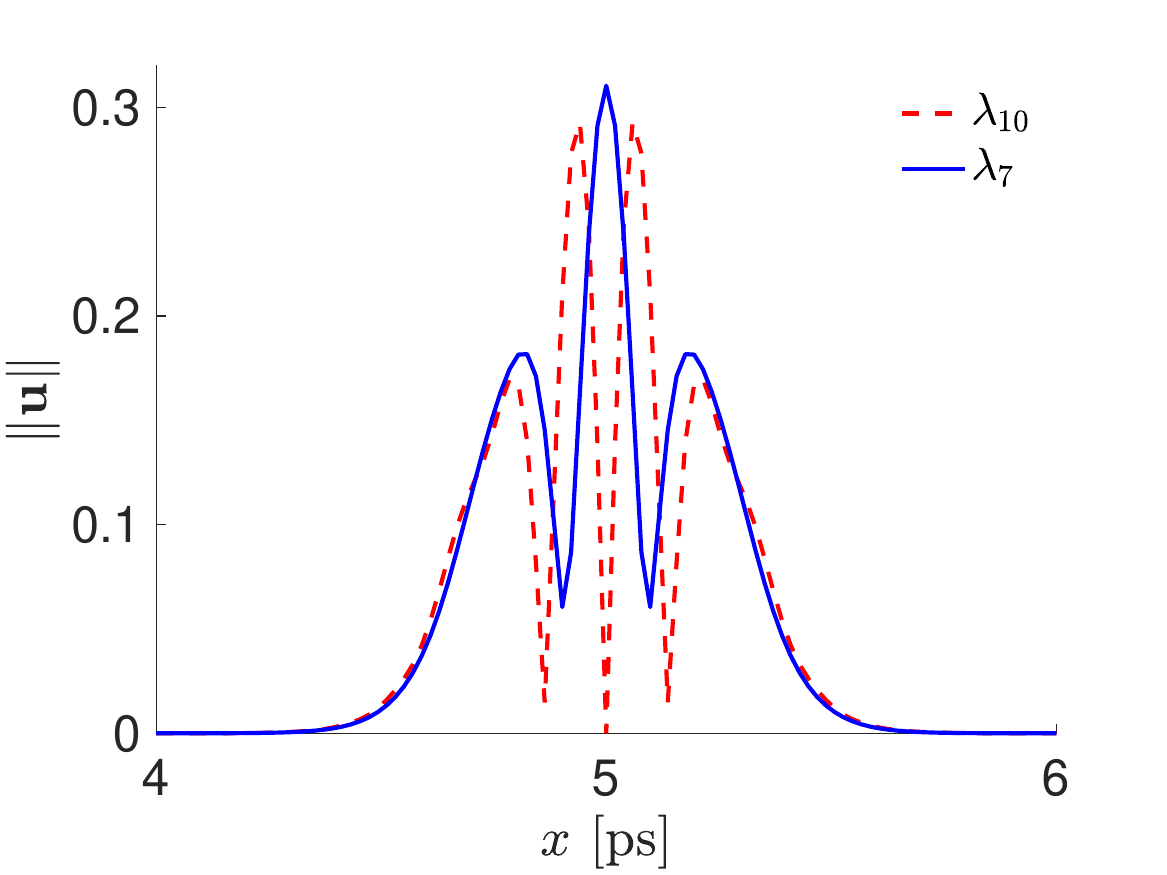}
            \includegraphics[width=0.49\linewidth]{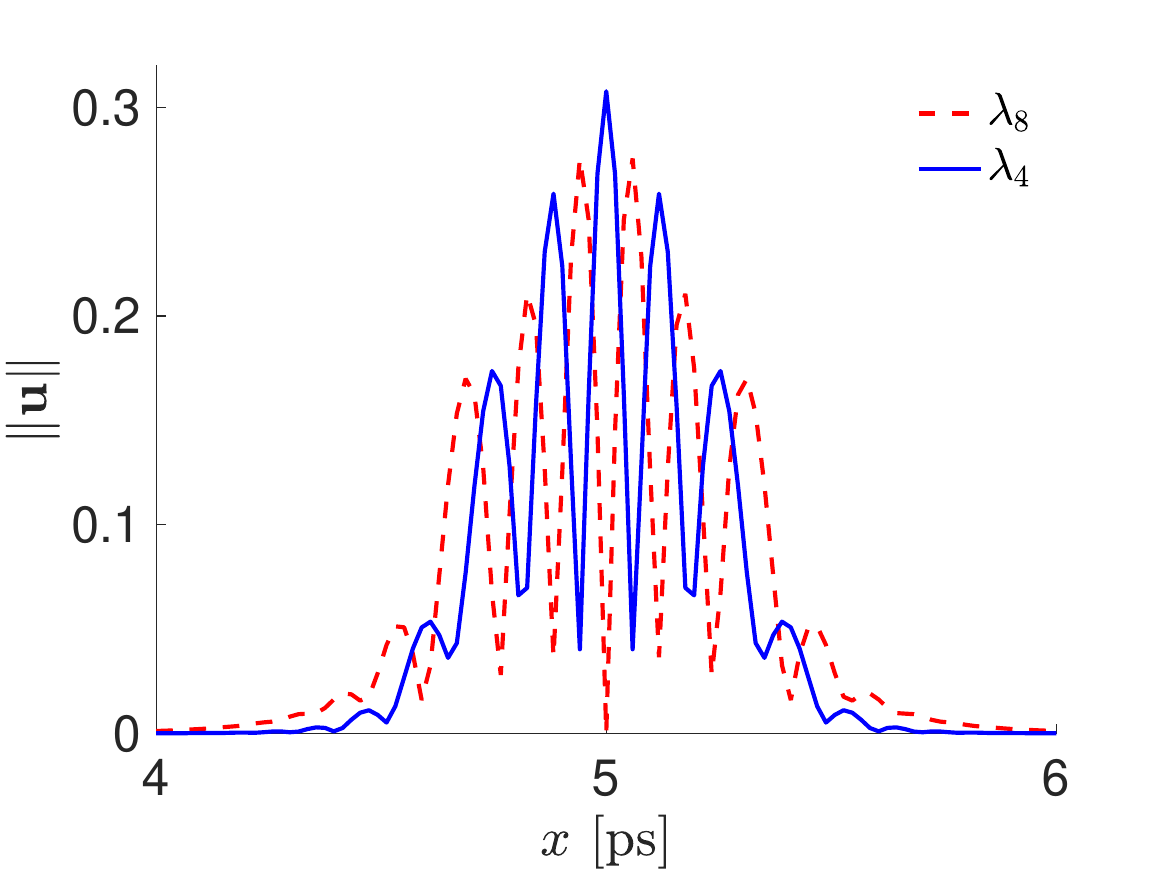}
    \caption{
       Left: Eigenfunctions corresponding to the  eigenvalues, $\lambda_{10} = 0.6040+0.6393i$  and $\lambda_7 =0.7711 + 0.4587i $.
    Right: Eigenfunctions corresponding to the  eigenvalues, $\lambda_8 = 0.4961 + 0.7397i$  and $\lambda_4 =0.5335 + 0.7379i $,
    that are closest and next to closest to the edge of the upper arm of the essential spectrum.
    %Right: Logarithmic scale plots of the eigenfunctions corresponding $\lambda_8$ and $\lambda_4$.
    }
    \label{fig:EigenfunctionsCloseFar}
\end{figure}

To investigate the extent of the region in parameter space where stable
pulses exist, we performed three parameter continuation studies.  
In~\cite{shinglot2022continuous}, we reported on how the parameters in the saturable absorber affect the essential spectrum. Here we focus on the parameters in the fiber amplifier.
Starting from the system parameters given above, we first increased
the unsaturated gain from $g_0=6$ to $g_0=7$ in increments of 
$0.1$. During this parameter continuation the peak power 
of the pulse increased linearly from 382~W to 493~W and the 
root mean square (RMS) pulse width increased linearly 
from 95~fs to 108~fs. In the left panel of \cref{fig:ParContSpecA} 
we show the essential spectrum at the final value $g_0=7$. 
In general, the edge of the upper arm of the essential spectrum is located at 
$\lambda_+(0)$, where $\lambda_+(\omega)$ is given in \eqref{eq:EssSpec2}. 
In particular, $|\lambda_+(0)|$ is determined by the balance of saturable 
gain and loss in the system, and $\arg(\lambda_+(0))=\theta$ is the optimized phase angle in \eqref{eq:Poincare map functional}. 
Just as for the standard soliton, as the peak power of the pulse increases,
(due to the increase in $g_0$)
the angle, $\theta$, increases, rotating the upper arm of the essential spectrum
counter-clockwise. In addition, the four complex 
eigenvalues in the first quadrant rotate in the same direction, 
approximately maintaining their distance from the unit circle.  
Significantly, at $g_0=7$, there is a fifth eigenvalue located just above 
$\lambda_+(0)$. This eigenvalue  bifurcates
out of the edge of essential spectrum when $g_0=6.5$.
Finally, the phase-like eigenvalue moves slightly in and the
translation-like eigenvalue does not move.

Next, returning to the original set of parameters, we increased the saturation energy from $E_{\text{sat}}=200$~pJ to 
$E_{\text{sat}}=260$~pJ in increments of $5$~pJ.
The peak power 
of the pulse increased linearly from 382~W to 461~W and the 
RMS pulse width increased linearly 
from 95~fs to 104~fs. In the right panel of \cref{fig:ParContSpecA} 
we show the essential spectrum at $E_{\text{sat}}=260$~pJ. 
Qualitatively, the same changes occur in  the spectrum  as when 
we increased $g_0$, except that the 
amount of rotation is not quite as large since the final peak power is lower.

Finally, we increased the fiber amplifier bandwidth from
$\Omega_g=50$~THz to $\Omega_g=145$~THz in increments of $0.5$~THz
and then jumped to $\Omega_g=500$~THz. 
During this parameter continuation the peak power 
 decreased from 382~W to 379~W at $\Omega_g=70$~THz 
and then increased to 382~W at $\Omega_g=500$~THz.
The RMS pulse width increased  from 95~fs to 110~fs.
In the left panel of \cref{fig:ParContSpecB} 
we show the essential spectrum at $\Omega_g=500$~THz.
Although the peak power does not change much, the wider filter
still results in a more nonlinear system, which results in $\theta$ increasing from 
$56^\circ$ to $67^\circ$. In addition, the essential spectrum spirals
much more slowly into the origin (we only show the first few rotations in the red curves). We also see that the
translation-like eigenvalue on the real axis has moved out to $\lambda = 0.999$,
while the phase-like eigenvalue moves inwards slightly, crossing the expanding
essential spectrum curve.
Meanwhile, the four discrete eigenvalues in the first quadrant move outward
towards the unit circle, slowing down significantly once $\Omega_g > 140$~THz.
In the right panel, we see that once $\Omega_g = 65$~THz a fifth eigenvalue has bifurcated out of the essential spectrum.
In the left panel of \cref{fig:OmegaJob3}, we see this eigenvalue 
 emerging from the essential spectrum at $\Omega_g=60$~THz, slightly behind the edge.

 \begin{figure}[!t]
    \centering
    \includegraphics[width=0.49\linewidth]{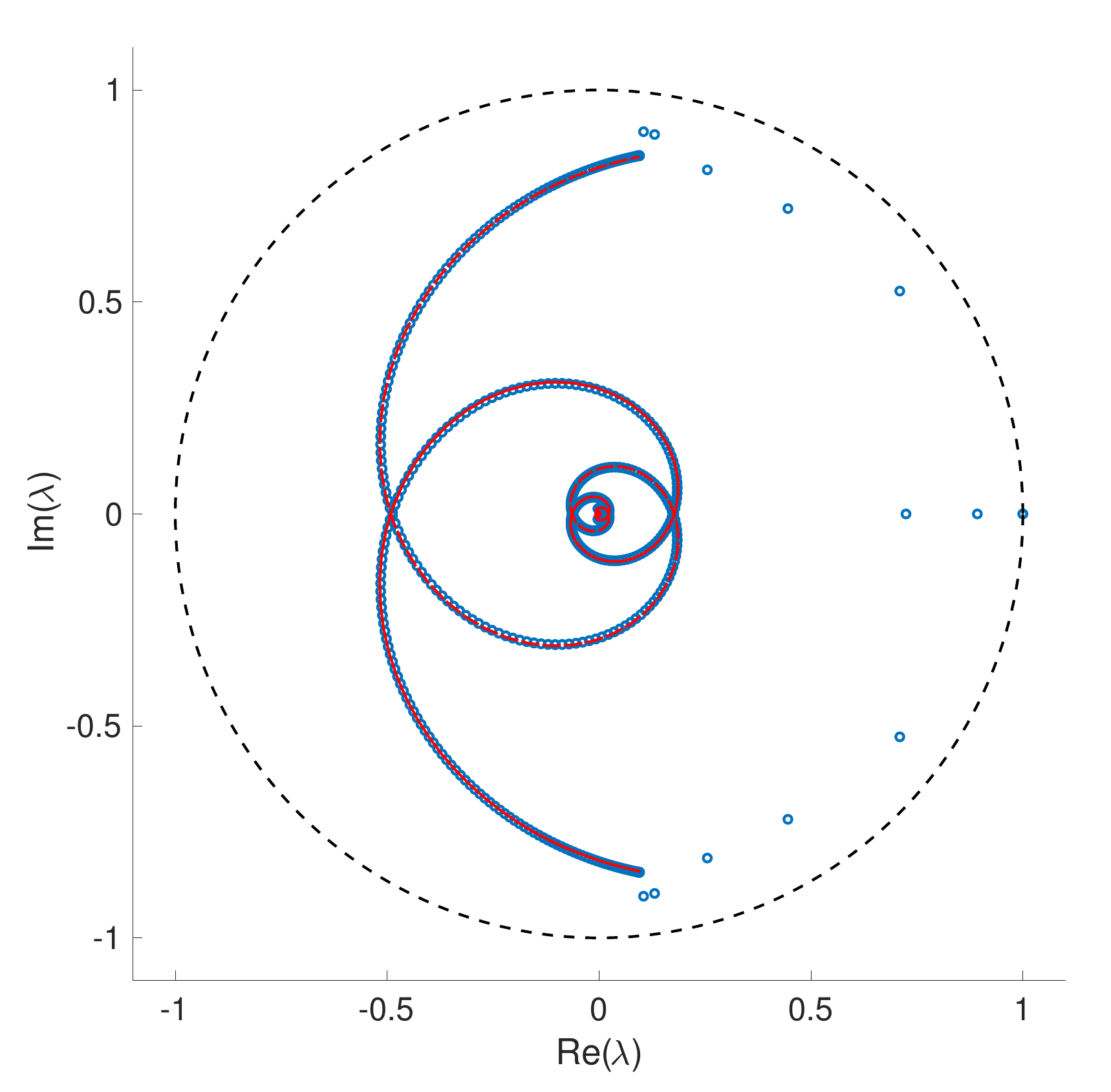}
   % \hskip20pt
    \includegraphics[width=0.49\linewidth]{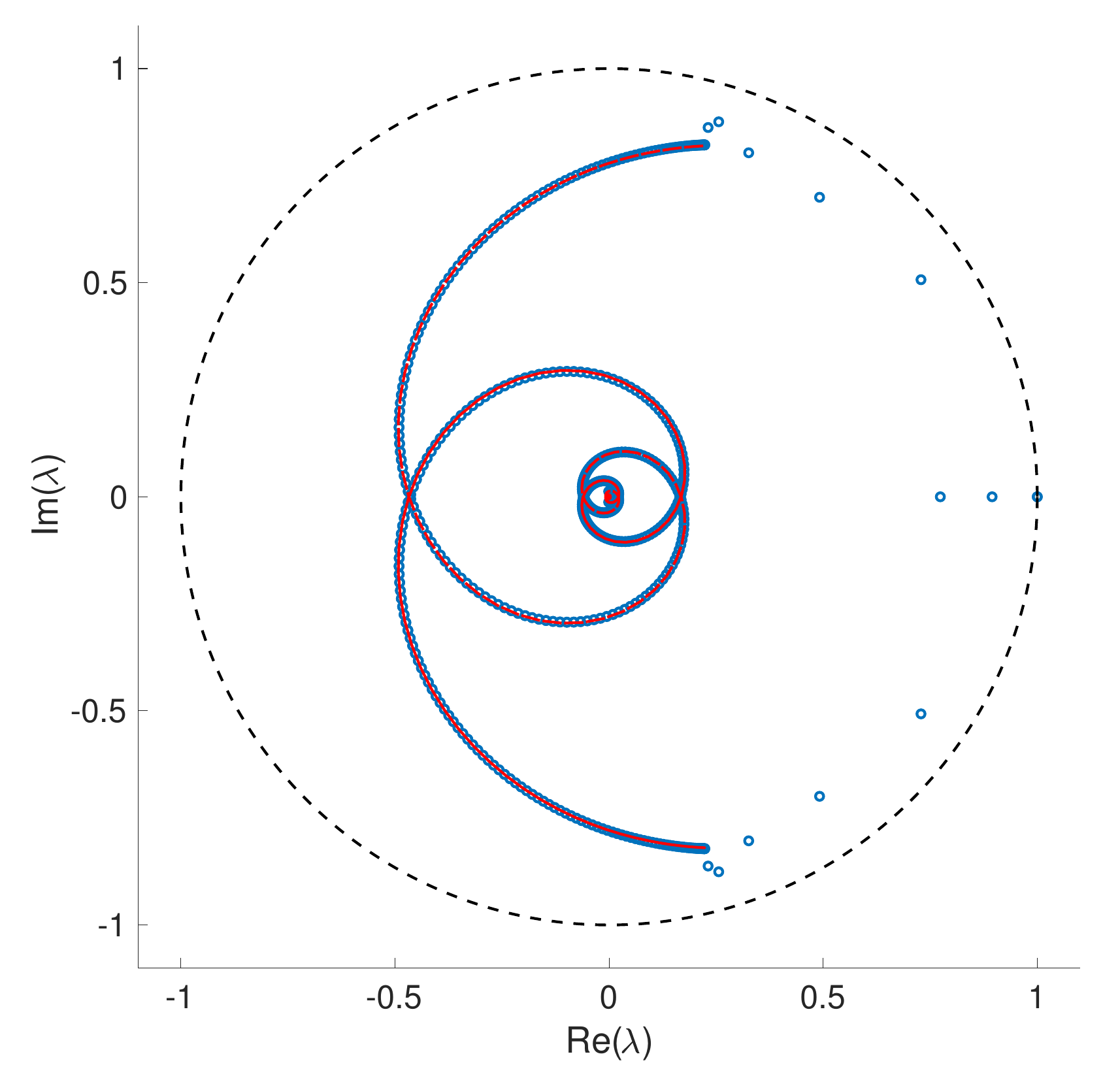}
    \caption{Spectra of the monodromy operator for $g_0=7$ (left) 
    and $E_{\text{sat}}=260$~pJ (right).}
    \label{fig:ParContSpecA}
\end{figure}

 \begin{figure}[!t]
    \centering
       \includegraphics[width=0.49\linewidth]{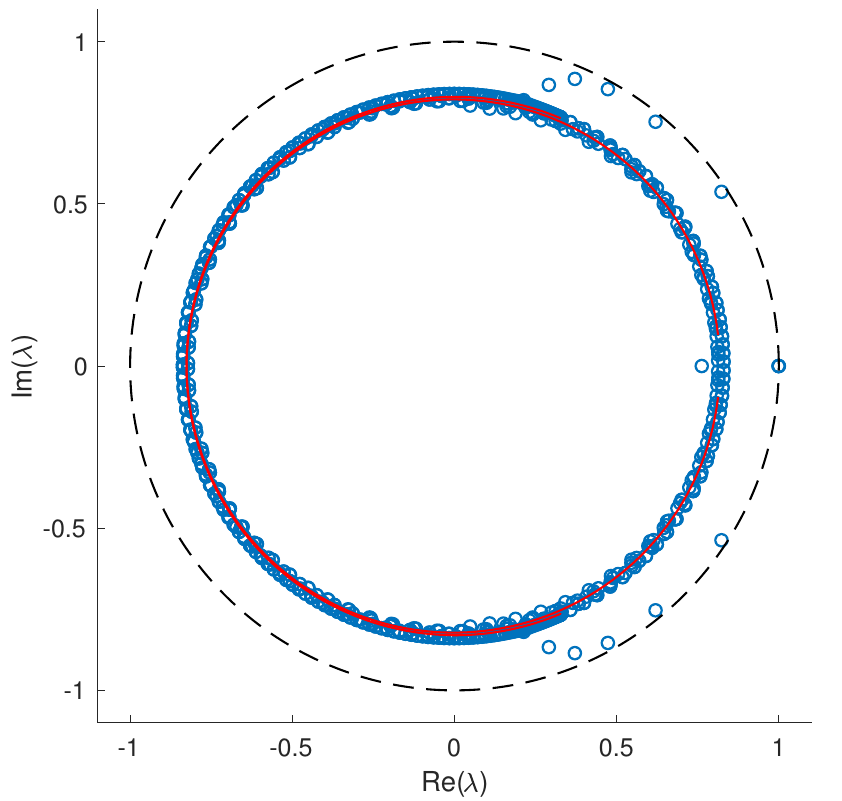}
             \includegraphics[width=0.49\linewidth]{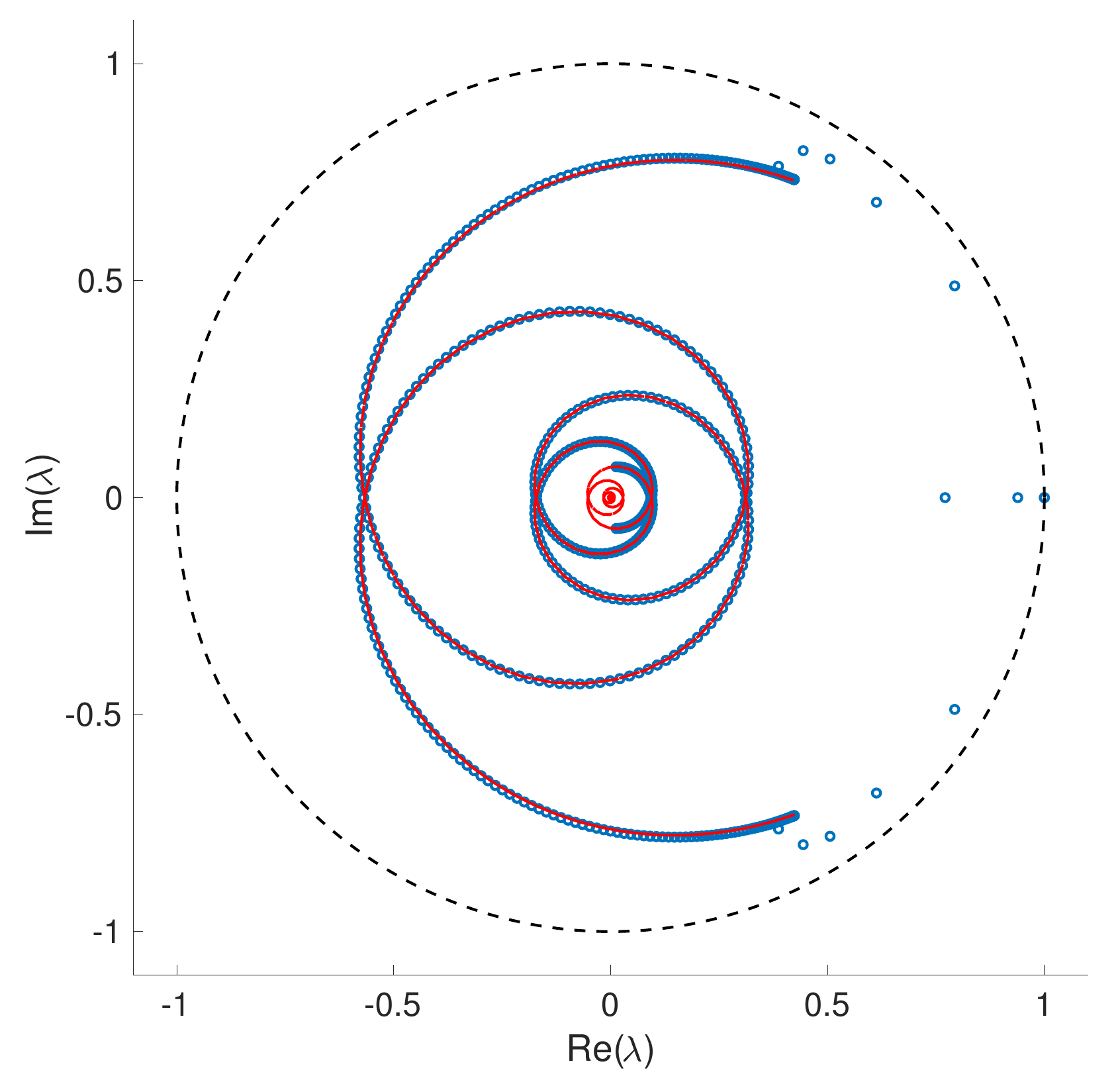}
    \caption{Spectra of the monodromy operator for $\Omega_g = 500$~THz (left) 
    and $\Omega_g = 65$~THz (right).}
    \label{fig:ParContSpecB}
\end{figure}

For stationary pulses, it is well known~\cite{barashenkov2000oscillatory,bridges2002stability} that there can be
significant errors when  the spectrum
of the linearized operator, $\mathcal L$, is approximated 
by the set of eigenvalues of a  matrix approximation, 
$\mathbf L$.
Specifically,  if an
 eigenfunction decays very slowly
there can be a large error in the
corresponding eigenvalue  due to windowing effects. 
This phenomenon only occurs for eigenvalues that are  sufficiently 
close to the essential spectrum.
In addition, the portion of the spectrum of $\mathbf L$
that corresponds to the essential spectrum may not agree with an
analytical formula for $\sigma_{\text{ess}}(\mathcal L)$.
For the eigenvalues of $\mathcal L$, the issue can be resolved by using computational Evans function
methods~\cite{bridges2002stability,SINUM53p2329} or by the iterative
solution of an appropriately formulated nonlinear eigenproblem~\cite{shen2016spectra,wang2014boundary}.
However, for periodically stationary pulses, even those obtained as solutions of
constant coefficient nonlinear wave equations, there is currently no 
numerical method for addressing this problem. 
Although it is no guarantee of accuracy, the best one can do is to double  the time window, $L$, and the number of points, $N$, and look for changes in  the location of the eigenvalues near the
essential spectrum and in the decay rates of the corresponding eigenfunctions.
Indeed, we verified that 
the location of the eigenvalue bifurcating out of the essential spectrum 
in the left panel of \cref{fig:OmegaJob3} does not change, and,
as we see in the right panel, neither does
the decay rate of the corresponding eigenfunction. 
In the left panel, we do however see some discrepancy between the 
analytic formula for the essential spectrum and its discrete approximation.
Similar differences occur near the edge of the essential spectrum for all
the simulations we performed. However, they are only evident on the larger
scale in the left panel of  \cref{fig:ParContSpecB}.
 
  \begin{figure}[!tbh]
    \centering
       \includegraphics[width=0.33\linewidth]{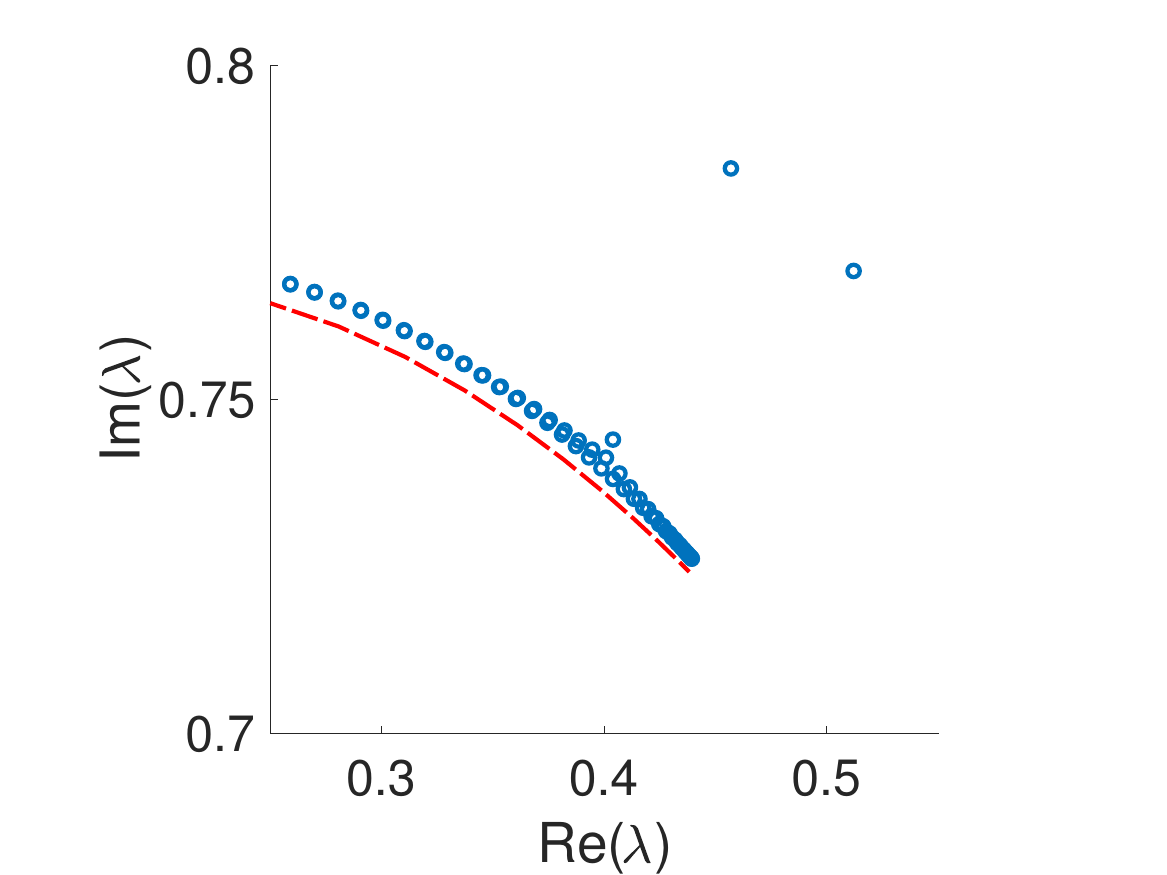}  
    \includegraphics[width=0.33\linewidth]{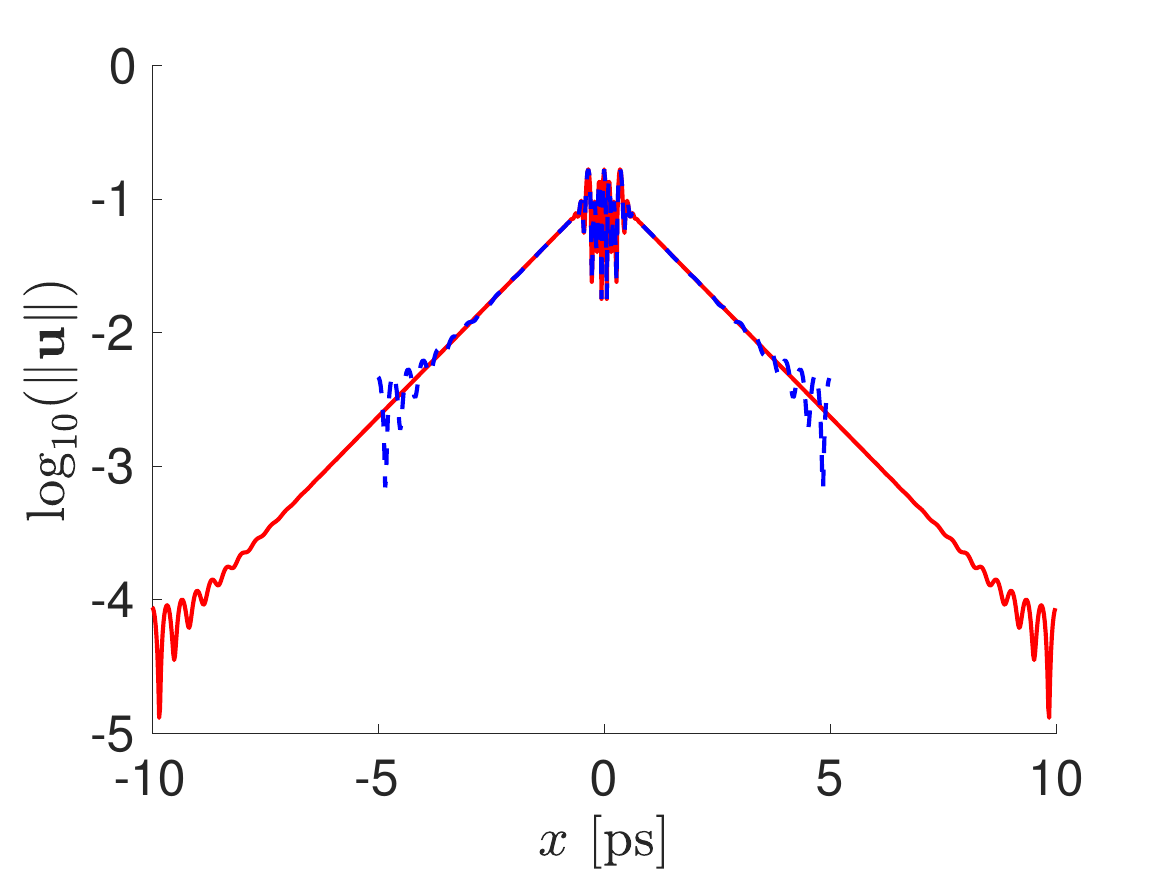}
    \caption{Left: Detail of the spectrum near the upper arm of the essential spectrum for $\Omega_g = 60~THz$. Right: Eigenfunction corresponding to
    the eigenvalue that bifurcates out of the essential spectrum on the left computed
    using $L=10$~ps, $N=512$ (blue dashed line) and $L=20$~ps, $N=1024$ (red solid line).}
    \label{fig:OmegaJob3}
\end{figure}

 \section{Conclusions}

 In this paper, we described and validated  accurate and efficient 
 computational methods 
 to discover periodically stationary pulses in a lumped model of a fiber laser
 and to assess their stability using the spectrum of  a monodromy operator. In particular, we demonstrated excellent agreement between the
 numerically computed spectrum on the one hand 
 and theoretical formulae for the essential
 spectrum and a multiplicity two eigenvalue on the other hand.  
Our simulations suggest that there is a large region in the parameter space 
of the fiber amplifier in which the Kim laser operates stably. An advantage
of the spectral approach to stability over the traditional evolution approach
used in the engineering community is that changes in the  spectrum 
can be used to predict the onset of an instability.
However, an unresolved theoretical problem is to establish a result relating spectral stability to linear stability in this context. 

To be useful for quantitative modeling of experimental lasers,
 the methods described here need to be extended to more realistic
 models of saturable absorbers (semiconductor saturable absorber mirrors)~\cite{menyuk2016spectral} and to erbium-doped fiber amplifiers modeled by multilevel rate equations~\cite{Giles1991ModelingEF}.
  In particular, we plan to apply our approach to the 
Mamyshev oscillator~\cite{rochette2008multiwavelength,sidorenko2018self,tarasov2016mode}, which has extremely large pulse variations
in which one half of the pulse is destroyed each round trip before being regenerated.
As is well appreciated by practitioners in the field, a major challenge of optimizing for  stationary and periodically stationary
pulses is the need for very good initial guesses. Further research on
parameter continuation methods  for pulse solutions of
nonlinear wave equations and lumped models is required to address this challenge~\cite{uecker2019hopf}. 
In addition, since the majority of computational time is devoted to computing
the monodromy matrix, it may prove advantageous to employ a
 matrix-free iterative method
to compute only the handful of eigenvalues that are not already identified by the
theory. 

A major challenge in the modeling of fiber lasers is to quantify the effects that 
quantum mechanical and technical noise sources have on the performance
of the system~\cite{paschotta2004noiseI,paschotta2004noiseII}. Traditionally this has been accomplished using theory for idealized models and highly
computationally intensive Monte Carlo simulations for more realistic ones. Building on classical
results of soliton perturbation theory, Menyuk has shown how to efficiently
quantify the system performance of stationary pulses in an averaged laser model 
by integration of the noise probability density function against numerically computed  eigenfunctions~\cite{menyuk2016spectral}. An important next step is to extend this approach to periodic stationary pulses.

\section*{Acknowledgments} We thanks  Y. Latushkin, J.  Marzuola, C.R.  Menyuk, and S. Wang for generously sharing their expertise with us.
The authors acknowledge computational resources provided by the Office of Information Technology Cyberinfrastructure Research Computing (CIRC) at The University of Texas at Dallas.
Finally, we thank the reviewers for their comments which improved the paper.

\bibliographystyle{siamplain}
\bibliography{VSTheory}
\end{document}